\documentclass[11pt]{article}
\usepackage{geometry}
\usepackage{moreverb}
\usepackage{graphicx}
\usepackage{caption,subcaption}
\usepackage{float}
\usepackage{amsfonts}
\usepackage{amsmath}
\usepackage{amsthm}
\usepackage{dsfont}
\usepackage{amssymb}
\usepackage{bbm}
\usepackage{booktabs}
\usepackage{threeparttable}
\usepackage[numbers,sort&compress]{natbib}
\usepackage{mathrsfs}
\usepackage{epstopdf}
\usepackage{mathtools}
\usepackage{bm}
\usepackage{relsize}
\usepackage{hyperref}
\hypersetup{
	colorlinks=true,
	citecolor=blue,
	linkcolor=blue,
}

\usepackage{authblk}
\usepackage{tikz}
\usetikzlibrary{arrows.meta, decorations.markings, calc}

\usepackage{booktabs} 
\usepackage{multirow}

\newcommand\keywords[1]{\textit{Keywords}: #1}
\newcommand\AMS[1]{\textit{MSC}: #1}

\theoremstyle{plain}
\newtheorem{theorem}{Theorem}[section]
\newtheorem{prop}{Proposition}[section]

\newtheorem{lemma}{Lemma}[section]

\theoremstyle{definition}

\newtheorem*{notations}{Notations}
\newtheorem{remark}{Remark}[section]

\theoremstyle{remark}

\numberwithin{equation}{section} 

\allowdisplaybreaks

\newcommand{\R}{\mathbb{R}}

\renewcommand{\d}{\mathrm{d}} 

\begin{document}
	
	\title{Strong ill-posedness for the MHD system in the supercritical regime: inviscid and viscous}
	\author[a]{Zhenyu Wan}
	\author[b]{Weikui Ye}
	\author[a]{Zhaoyang Yin}
	\affil[a]{School of Science, Shenzhen Campus of Sun Yat-Sen University, Shenzhen, 518107, China}
	\affil[b]{School of Mathematical Sciences, Shenzhen University, Shenzhen, 518060, China}

	\renewcommand*{\Affilfont}{\small\it} 
	\renewcommand\Authands{, }
	\date{}
	\maketitle

    \begingroup
    \renewcommand\thefootnote{}
    \footnotetext{Email addresses:
	wanzhy9@mail2.sysu.edu.cn (Z. Wan),
	904817751@qq.com (W. Ye),
	mcsyzy@mail.sysu.edu.cn (Z. Yin).}%
    \addtocounter{footnote}{-1}
    \endgroup
	
	\begin{abstract}
	This paper is concerned with the Cauchy problem for the 3D incompressible magnetohydrodynamic (MHD) equations in supercritical Sobolev spaces. It is well known that the system is locally well-posed in subcritical Sobolev spaces, whereas the supercritical regime remains largely open. In this work, we establish norm inflation for the incompressible MHD equations, both with and without Laplacian dissipation, in supercritical Sobolev spaces, thereby revealing strong ill-posedness of the system at this regularity level. A distinctive feature of our approach is the introduction of a novel geometric construction, termed the ``Magnetic-solo ansatz'', through which, for the ideal MHD system, norm inflation occurs exclusively in the magnetic field $b$ in $H^s$ with $0<s<\frac{5}{2}$, while the $H^s$-norm of the velocity field $u$ remains uniformly bounded. This asymmetric behavior shows that supercritical ill-posedness can be driven exclusively by the magnetic field, highlighting its essential role in the breakdown of well-posedness. Our findings fill a significant gap in the supercritical regularity theory for incompressible MHD and shed light on the distinct mechanisms governing the fluid and magnetic dynamics.
	\end{abstract}

	\AMS{35Q35, 35B44, 35R25}
	

	\section{Introduction}
	The dynamics of electrically conducting fluids interacting with magnetic fields are governed by the incompressible magnetohydrodynamics (MHD) equations. This system is of fundamental physical importance and arises in a wide range of applications, from magnetic reconnection in astrophysical plasmas to plasma confinement in thermonuclear fusion devices (see, e.g., \cite{biskamp1993nonlinear,priest2000magnetic}). 
	
	In this paper, we study the following incompressible MHD equations in $\R^3$:
	\begin{equation}\label{mhd-4}
	\begin{cases}
	\partial_t u -\varepsilon_1 \Delta u + u \cdot \nabla u + \nabla p = b \cdot \nabla b, \\
	\partial_t b - \varepsilon_2 \Delta b+ u \cdot \nabla b = b \cdot \nabla u, \\
	\text{div } u = \text{div } b = 0,
	\end{cases}
	\end{equation}
	where $u=(u_1,u_2,u_3)$ denotes the velocity field, $b=(b_1,b_2,b_3)$ represents the magnetic field, and $p$ is the scalar pressure. The parameters $\varepsilon_1\ge 0$ and $\varepsilon_2\ge 0$ represent the  viscosity coefficient and magnetic diffusivity, respectively.
	
	According to whether the viscosity and magnetic diffusion are present, system \eqref{mhd-4} reduces to four distinct cases. More precisely, we will systematically investigate the following four systems:
	
	\vspace{2mm} 
	\noindent \textbf{(I) The ideal MHD equations} ($\varepsilon_1 = \varepsilon_2= 0$):
    \begin{equation}\begin{cases}\label{ideal mhd}
    \partial_{t}u +u\cdot \nabla u+ \nabla p = b\cdot \nabla b ,\\
    \partial_{t}b + u\cdot \nabla b = b\cdot \nabla u,\\
    \operatorname{div}u = \operatorname{div}b = 0.
    \end{cases}
    \end{equation}
    
    \noindent \textbf{(II) The non-resistive MHD equations} ($\varepsilon_1=1, \varepsilon_2=0$):
    \begin{equation}\begin{cases}\label{non-resistive mhd}
    \partial_{t}u -\Delta u+ u\cdot \nabla u+ \nabla p = b\cdot \nabla b ,\\
    \partial_{t}b + u\cdot \nabla b = b\cdot \nabla u,\\
    \operatorname{div}u = \operatorname{div}b = 0.
    \end{cases}
    \end{equation}
    
    \noindent \textbf{(III) The non-viscous MHD equations} ($\varepsilon_1= 0, \varepsilon_2=1$):
    \begin{equation}\begin{cases}\label{non-viscous mhd}
    \partial_{t}u + u\cdot \nabla u+ \nabla p = b\cdot \nabla b ,\\
    \partial_{t}b -\Delta b+ u\cdot \nabla b = b\cdot \nabla u,\\
    \operatorname{div}u = \operatorname{div}b = 0.
    \end{cases}
    \end{equation}

    \noindent \textbf{(IV) The viscous and resistive MHD equations} ($\varepsilon_1= \varepsilon_2=1$):
    \begin{equation}\begin{cases}\label{viscous-resistive mhd}
    \partial_{t}u -\Delta u+ u\cdot \nabla u+ \nabla p = b\cdot \nabla b ,\\
    \partial_{t}b -\Delta b+ u\cdot \nabla b = b\cdot \nabla u,\\
    \operatorname{div}u = \operatorname{div}b = 0.
    \end{cases}
    \end{equation}
    
    	\subsection{Background: The ideal MHD equations}
    The well-posedness theory and blow-up criteria for the ideal MHD equations have been extensively studied. For instance, Beale--Kato--Majda type blow-up criteria for smooth solutions have been established via losing estimates \cite{cannone2007losing}. Regarding the well-posedness threshold, Chen, Miao, and Zhang \cite{chen2010well} established local well-posedness and blow-up criteria in the Triebel--Lizorkin spaces $F^s_{p,q}(\R^d)$ for $s>\frac{d}{p}+1$. By the identification $F^s_{2,2}(\R^3)=H^s(\R^3)$, their result yields classical local well-posedness in $H^s(\R^3)\times H^s(\R^3)$ for $s>\frac{5}{2}$.
    
    When the magnetic field vanishes $(b=0)$, the system reduces to the incompressible Euler equations. For the Euler equations, strong ill-posedness in critical Sobolev spaces was established by Bourgain and Li \cite{bourgain2015strong}. Recently, Luo \cite{luo2024illposedness} established supercritical ill-posedness for the 3D Euler equations in $0<s<\frac{5}{2}$ utilizing a vortex ring.
    
    However, for the ideal MHD system, the presence of the magnetic field gives rise to significantly more intricate nonlinear coupling effects. While the global well-posedness for the ideal MHD equations has been widely investigated under specific background magnetic fields \cite{cai2018global,he2018global,wei2017global,wei2018global}, the behavior of strong solutions below the classical Sobolev threshold remains unresolved. \textbf{In view of the local well-posedness result of Chen, Miao, and Zhang \cite{chen2010well}, it has been an open problem whether the ideal MHD equations remain well-posed in the supercritical Sobolev regime below $s=\frac{5}{2}$. In this paper, we give a negative answer to this question by proving strong ill-posedness in $H^s(\R^3) \times H^s(\R^3)$ for $0<s<\frac{5}{2}$.}

    \subsection{Background: The dissipative variants of MHD equations}
    The study of the MHD equations under various combinations of kinematic viscosity and magnetic diffusivity has been the subject of extensive research.
    \subsubsection{Non-resistive MHD}
   When the magnetic field vanishes $(b=0)$, system
   \eqref{non-resistive mhd} reduces to the incompressible
   Navier--Stokes equations. Bourgain and Pavlovi\'{c} \cite{bourgain2008ill} established norm inflation in the largest critical space $\dot{B}^{-1}_{\infty,\infty}$. More recently, Luo \cite{luo2024illposedness,luo2025sharp} established strong ill-posedness for the 3D Navier--Stokes equations in $H^s$ for $0<s<\frac12$ and in almost all supercritical homogeneous Besov spaces $\dot{B}^s_{p,q}$ near the scaling-critical line. In a related recent work, Chen and Xie \cite{chen2026norm} established norm inflation for the two-dimensional inviscid and fully dissipative Boussinesq systems in supercritical Besov spaces.  When the magnetic field is coupled, the lack of magnetic diffusion makes the non-resistive MHD system essentially hyperbolic regarding the magnetic field, introducing severe analytical challenges. The well-posedness theory of \eqref{non-resistive mhd} has been extensively studied. For instance, global classical solutions in 3D for small initial data were established in \cite{xu2015global,abidi2017global}. Concerning local well-posedness, Fefferman et al. \cite{fefferman2014higher,fefferman2017local} proved local existence and uniqueness. Indeed, via time-space mixed Besov spaces, one can generalize the well-posedness result to $H^s(\R^3)\times H^{s+1}(\R^3)$ for $s>\frac{1}{2}$. This result was further extended to the homogeneous Besov spaces $\dot{B}^{\frac{d}{p}-1}_{p,1}(\R^d) \times \dot{B}^{\frac{d}{p}}_{p,1}(\R^d)$ for $1\le p\le2d$ and $d\ge 2$ by Li, Tan, and Yin \cite{li2017local}. Moreover, for $d=2,3$, Li, Yin, and Zhu \cite{li2023non} showed that the data-to-solution map in $H^s(\R^d)\times H^s(\R^d)$ for $s>\frac d2$, is continuous but not uniformly continuous. Very recently, Chen, Nie, and Ye \cite{chen2024sharp} established the sharp ill-posedness in critical Sobolev spaces.

    Among the above works, the nearly optimal local well-posedness result of Fefferman et al. \cite{fefferman2017local} is particularly representative. \textbf{In the conclusion of their paper, they wrote that \textit{``it would be interesting to find a simpler model problem in which it is possible to demonstrate the failure of local existence for $b_0\in H^{s+1}$ and $u_0\in H^s$.''} In this paper, we provide a partial answer to the problem posed by Fefferman et al. in the supercritical regime by proving strong ill-posedness via norm inflation in $H^s(\R^3)\times H^{s+1}(\R^3)$ for $0<s<\frac12$.}

    \subsubsection{Non-viscous MHD}
    This system models resistive plasmas where fluid viscosity is negligible. In two dimensions, Wei and Zhang \cite{wei2020global} established the first global well-posedness result for the non-viscous MHD equations \eqref{non-viscous mhd} with small initial data. For the 3D case, Hassainia \cite{hassainia2022global} remarkably proved the global well-posedness for the MHD equations with axisymmetric initial data. It is also well known that the system is locally well posed in $H^{s}(\R^3)\times H^{s-1}(\R^3)$ for $s>\frac{5}{2}$, which can be established by adapting the argument used in Proposition 1.3 of Wu and Zhao \cite{wu2023mild}.

    Despite these positive results, the behavior of solutions below this regularity threshold remains largely open. \textbf{Recently, Wu and Zhao \cite{wu2023mild} established mild ill-posedness in two dimensions. However, in three dimensions, it has been an open problem whether the system is well-posed in supercritical Sobolev spaces. In this paper, we give a negative answer by proving strong ill-posedness in $H^s(\R^3)\times H^{s-1}(\R^3)$ for $1<s<\frac52$.}

    \subsubsection{Viscous and resistive MHD}
    In this system, both kinematic viscosity and magnetic diffusion are present. The local and global well-posedness theory has been extensively developed since the pioneering works of Duvaut and Lions \cite{duvaut1972inequations} and Sermange and Temam \cite{sermange1983some}. In particular, local well-posedness is classically known to hold in $H^s(\R^3)\times H^s(\R^3)$ for $s\ge\frac{1}{2}$. Furthermore, the existence of Leray-Hopf weak solutions to the MHD equations was proved by Wu \cite{wu2003generalized}. For strong solutions, Beale--Kato--Majda type blow-up criteria were subsequently established by Chen, Miao, and Zhang \cite{chen2007beale} via the velocity vorticity only.

    On the other hand, the non-uniqueness of weak solutions to this system has been intensely studied via the convex integration method. Remarkably, Li, Zeng, and Zhang \cite{li2022non} established the non-uniqueness of weak solutions for the 3D viscous and resistive MHD equations, including the hyper viscous and resistive case up to the Lions exponent $\frac{5}{4}$. Subsequently, the authors \cite{li2024sharp} extended this non-uniqueness result to the hyper viscous and resistive MHD beyond the Lions exponent.
   
    While these works reveal ill-posedness through the non-uniqueness of weak solutions, the occurrence of norm inflation for strong solutions in supercritical Sobolev spaces reflects a fundamentally different mechanism. The classical well-posedness and weak-solution theory was developed by Duvaut and Lions \cite{duvaut1972inequations}, Sermange and Temam \cite{sermange1983some}, and Wu \cite{wu2003generalized}. \textbf{However, it has remained unsolved whether the system \eqref{viscous-resistive mhd} is well-posed in supercritical Sobolev spaces. In this paper, we also give a negative answer by proving strong ill-posedness in $H^s(\R^3)\times H^s(\R^3)$ for $0<s<\frac12$.}

     The presence or absence of the dissipative terms $\Delta u$ and $\Delta b$ leads to different Sobolev regularity thresholds for local well-posedness. We summarize the corresponding indices $(s_u,s_b)$ for the four 3D MHD systems in Table \ref{tab:mhd_critical_indices}.
     
     \begin{table}[H]
     	\centering
     	\renewcommand{\arraystretch}{1.5} 
     	\begin{tabular}{@{}l c c c c@{}}
     		\toprule
     		\multirow{2}{*}{\textbf{MHD System}} & \multicolumn{2}{c}{\textbf{Dissipative Terms}} & \multicolumn{2}{c}{\textbf{Critical Indices}} \\
     		\cmidrule(lr){2-3} \cmidrule(l){4-5} 
     		& \textbf{Velocity} & \textbf{Magnetic} & $\mathbf{s_u}$ & $\mathbf{s_b}$ \\
     		\midrule
     		Ideal MHD             & $0$        & $0$        & $\frac{5}{2}$ & $\frac{5}{2}$ \\
     		Non-resistive MHD     & $\Delta u$ & $0$        & $\frac{1}{2}$ & $\frac{3}{2}$ \\
     		Non-viscous  MHD& $0$        & $\Delta b$ & $\frac{5}{2}$ & $\frac{3}{2}$ \\
     		Viscous and resistive MHD & $\Delta u$ & $\Delta b$ & $\frac{1}{2}$ & $\frac{1}{2}$ \\
     		\bottomrule
     	\end{tabular}
     	\vspace{0.2cm}
     	\caption{Critical Sobolev indices $(s_u, s_b)$ for the local well-posedness of the 3D MHD systems in $H^{s_u}(\R^3)\times H^{s_b}(\R^3)$ under different combinations of dissipative terms.}
     	\label{tab:mhd_critical_indices}
     \end{table}

     \subsection{Main results}
 
     As shown in Table \ref{tab:mhd_critical_indices}, the equations are locally well-posed above their respective critical thresholds. However, it has remained an open problem whether the 3D incompressible MHD equations are still well-posed in the supercritical Sobolev regimes below these thresholds.

     Our main results provide a negative answer to this question by establishing strong ill-posedness in supercritical Sobolev spaces for both the ideal and dissipative MHD equations. We first state the main result for the ideal case.
     \begin{theorem}\label{th-ill-ideal}
    	The 3D ideal MHD equation \eqref{ideal mhd} is strongly ill-posed in $H^{s}(\R^3)\times H^{s}(\R^3)$ for any $0<s<\frac{5}{2}$ in the following sense. 
    	
    	For any $\varepsilon>0$, there exists a family of initial data  $(u_{0,\varepsilon},b_{0,\varepsilon})\in C_c^\infty(\R^3)\times C_c^\infty(\R^3)$ such that the following hold.
    \begin{enumerate}
   		\item[(1)] $u_{0,\varepsilon}, b_{0,\varepsilon}$ are divergence-free and \[\left\|u_{0,\varepsilon}\right\|_{H^{s}(\R^3)}+\left\|b_{0,\varepsilon}\right\|_{H^{s}(\R^3)}\le\varepsilon.\]
    	\item[(2)] There exists a time $0<t^*\le\varepsilon$ such that $(u_{\varepsilon},b_{\varepsilon})\in C([0,t^*];H^{\infty}(\R^3)\times H^{\infty}(\R^3))$ and \[\left\|b_{\varepsilon}(t^*)\right\|_{H^{s}(\R^3)}\ge\frac{1}{\varepsilon}.\]
    	 \item[(3)] Simultaneously, the velocity field remains uniformly bounded without any inflation
    	\begin{equation*}
    	\|u_\varepsilon\|_{L^\infty([0, t^*]; H^s(\R^3))} \le \varepsilon.
    	\end{equation*}
    \end{enumerate}
    \end{theorem}
	
	\begin{remark}
		Unlike the recent norm inflation results for the Euler and Navier--Stokes equations \cite{luo2024illposedness}, where the inflation occurs in the velocity field, our construction for the MHD system exhibits a fundamentally different mechanism: the velocity field remains uniformly bounded, while the magnetic field undergoes norm inflation. This highlights the distinct role of the velocity field, which acts primarily as a steady transport flow driving the magnetic amplification.
	\end{remark}

Remarkably, our method of proving Theorem \ref{th-ill-ideal} is robust enough that only minor modifications are needed to handle the dissipative cases. We show that this strong ill-posedness persists across the three variants \eqref{non-resistive mhd}-\eqref{viscous-resistive mhd}, providing a unified treatment of the ill-posedness theory for the entire family of 3D incompressible MHD equations in the corresponding supercritical Sobolev spaces.
	\begin{theorem}\label{th-ill-mhd-unified}
		The strong ill-posedness persists for the other three variants of the 3D MHD equations in $H^{s_u}(\R^3) \times H^{s_b}(\R^3)$. Specifically, we define the Sobolev indices $(s_u, s_b)$ for the respective systems as follows:
		\begin{itemize}
			\item For the non-resistive MHD equation \eqref{non-resistive mhd}: $(s_u, s_b) = (s, s+1)$ with $0 < s < \frac{1}{2}$;
			\item For the non-viscous MHD equation \eqref{non-viscous mhd}: $(s_u, s_b) = (s, s-1)$ with $1 < s < \frac{5}{2}$;
			\item For the viscous and resistive MHD equation \eqref{viscous-resistive mhd}: $(s_u, s_b) = (s, s)$ with $0 < s < \frac{1}{2}$.
		\end{itemize}
		
		For each of the cases above, the corresponding equation is strongly ill-posed in the following sense: for any $\varepsilon>0$, there exists a family of initial data $(u_{0,\varepsilon}, b_{0,\varepsilon}) \in C_c^\infty(\R^3) \times C_c^\infty(\R^3)$ such that the following hold.
		\begin{enumerate}
			\item[(1)] $u_{0,\varepsilon}$ and $b_{0,\varepsilon}$ are divergence-free and 
			\begin{equation*}
			\left\|u_{0,\varepsilon}\right\|_{H^{s_u}(\R^3)} + \left\|b_{0,\varepsilon}\right\|_{H^{s_b}(\R^3)} \le \varepsilon.
			\end{equation*}
			\item[(2)] There exists a time $0 < t^* \le \varepsilon$ such that $(u_{\varepsilon}, b_{\varepsilon}) \in C([0,t^*]; H^{\infty}(\R^3) \times H^{\infty}(\R^3))$ and 
			\begin{equation*}
			\left\|u_{\varepsilon}(t^*)\right\|_{H^{s_u}(\R^3)} +\left\|b_{\varepsilon}(t^*)\right\|_{H^{s_b}(\R^3)} \ge \frac{1}{\varepsilon}.
			\end{equation*}
		\end{enumerate}
	\end{theorem}
	\begin{remark}
		We emphasize that the inflation mechanism differs between the three dissipative variants. For the non-resistive MHD equations \eqref{non-resistive mhd} and the viscous-resistive MHD equations \eqref{viscous-resistive mhd}, the norm inflation occurs in the magnetic field $b$, while the velocity field $u$ remains uniformly bounded, as in the ideal MHD case. In contrast, for the non-viscous MHD equations \eqref{non-viscous mhd}, the norm inflation occurs in the velocity field $u$, whereas the magnetic field remains uniformly bounded up to the critical time $t^*$.
	\end{remark}

    To provide a systematic overview of our results, we contrast the classical local well-posedness theories with our strong ill-posedness results in supercritical Sobolev spaces for the four 3D MHD systems in Table \ref{tab:summary_results}.

    \begin{table}[H] 
    	\centering
    	\renewcommand{\arraystretch}{1.5}
    	\resizebox{\linewidth}{!}{%
    		\begin{tabular}{l c r}
    			\toprule
    			\textbf{MHD System}
    			& \textbf{Local Well-posedness}
    			& \textbf{Ill-posedness in
    				$H^{s_u}(\R^3)\times H^{s_b}(\R^3)$} \\
    			\midrule
    			Ideal MHD 
    			& \cite{majda1984compressible,chen2010well} 
    			& $s_u=s_b=s,\ 0<s<\frac52$,
    			\quad Theorem \ref{th-ill-ideal} \\
    			
    			Non-resistive MHD 
    			& \cite{fefferman2017local} 
    			& $(s_u,s_b)=(s,s+1),\ 0<s<\frac12$,
    			\quad Theorem \ref{th-ill-mhd-unified} \\
    			
    			Non-viscous MHD 
    			& \cite{wu2023mild}
    			& $(s_u,s_b)=(s,s-1),\ 1<s<\frac52$,
    			\quad Theorem \ref{th-ill-mhd-unified} \\
    			
    			Viscous and resistive MHD 
    			& \cite{duvaut1972inequations,sermange1983some,wu2003generalized} 
    			& $s_u=s_b=s,\ 0<s<\frac12$,
    			\quad Theorem \ref{th-ill-mhd-unified} \\
    			\bottomrule
    		\end{tabular}%
    	}
    	\vspace{0.2cm}
    	\caption{Comparison of the local well-posedness and the
    		supercritical strong ill-posedness regimes in Sobolev spaces
    		proven in this paper.}
    	\label{tab:summary_results}
    \end{table}
	
   \begin{remark}
   	With suitable modifications of the Magnetic-solo ansatz to be introduced below and the construction developed by Luo in \cite{luo2025sharp}, we expect that our strong ill-posedness results for all four MHD systems can be extended to Besov spaces at the corresponding supercritical regimes, which is still an interesting problem.
   \end{remark}

	\subsection{Main strategy and construction}
	A central question addressed in this paper is: Why can a single unified framework induce norm inflation across four dynamically distinct MHD systems? 
	
	To overcome the intricate nonlinear couplings between the velocity and magnetic fields, we introduce a carefully designed anisotropically decoupled approximate evolution, which we call the \textbf{Magnetic-solo ansatz}.

	Let $(r,\theta,z)$ denote the standard cylindrical coordinates in $\R^3$. More precisely, as defined in \eqref{approx_sol}, the Magnetic-solo ansatz is given by an approximate solution $(\bar u,\bar b)$ consisting of a purely poloidal, swirl-free velocity field and a purely toroidal magnetic field:
	\begin{equation*}
	\bar{u}(t,x) = u_{0,r}\mathbf{e}_r + (u_{0,z} + u_c)\mathbf{e}_z, \quad \bar{b}(t,x) = \bar{b}_\theta\mathbf{e}_\theta.
	\end{equation*}
	Here, the toroidal component $\bar b_\theta$ is transported by the stationary poloidal velocity according to
	\begin{equation*}
	\partial_t\bar b_\theta
	+
	\bigl(u_{0,r}\partial_r+u_{0,z}\partial_z\bigr)\bar b_\theta
	=
	0,
	\end{equation*}
	as specified in \eqref{transport_b}. Under this specific geometry, the velocity field $\bar{u}$ acts as a stationary 2D incompressible Euler vortex ring and remains uniformly bounded. It serves purely as a steady shear flow that continuously drives the magnetic field, producing rapid norm inflation solely in the magnetic component $\bar{b}$. This geometric decoupling allows the pair $(\bar u,\bar b)$ to satisfy the approximate MHD system \eqref{error_eqs} with controlled residual errors. See Section \ref{sec-approx} for details.

	Our geometric reduction highlights a fundamental difference between the MHD system and purely hydrodynamic models. In the recent strong ill-posedness results for the Euler and Navier-Stokes equations by Luo \cite{luo2024illposedness}, norm inflation occurs in the velocity field. In contrast, under the Magnetic-solo ansatz, the velocity field remains uniformly bounded.

	This reduction enables us to isolate the leading-order dynamics from the complex nonlinear couplings. To establish norm inflation for all four systems, we discover two distinct mechanisms depending on the presence or absence of specific dissipative terms.

	\vspace{2mm}
	\noindent \textbf{Mechanism I:} Magnetic-solo inflation (for ideal, non-resistive, and viscous and resistive MHD)
	
	For equations \eqref{ideal mhd}, \eqref{non-resistive mhd}, and \eqref{viscous-resistive mhd}, the Magnetic-solo ansatz produces magnetic norm inflation through the linear transport-stretching mechanism. Driven by the steady velocity shear, the magnetic field undergoes rapid norm inflation in $H^{s_b}$. Crucially, since the critical inflation time $t^*\ll1$, the dissipative terms $-\Delta u$ and $-\Delta b$, whenever present, remain negligible at leading order and can be treated as perturbative errors.
	
	\vspace{2mm}
	\noindent \textbf{Mechanism II:} Velocity inflation via Euler dynamics (for non-viscous MHD)
	
	For the non-viscous MHD equations \eqref{non-viscous mhd}, the Magnetic-solo ansatz does not produce the desired magnetic inflation. We therefore reduce the amplitude of the magnetic field and regard it as a perturbation of the Euler dynamics, so that norm inflation occurs in the velocity field while the magnetic field remains uniformly bounded.
	
	\vspace{2mm}
	\noindent\textbf{Analytical Challenge: Excluding blow-up before $t^*$.}
	
	A fundamental difficulty in establishing supercritical norm inflation lies in the time scale. Classical local well-posedness theory guarantees a lifespan $T\sim\|\nabla(u_0,b_0)\|_{L^\infty}^{-1}$. In our construction, however, the critical time $t^*$ at which norm inflation occurs exceeds this classical lifespan by a factor of $\varepsilon^{-N}$. The main challenge is to prove that the exact solution $(u,b)$ excludes finite-time blow-up up to $t^*$ while remaining sufficiently close to the approximate solution $(\bar u,\bar b)$ in $H^s$. 
	
	To overcome this difficulty, we develop a direct energy estimate for the perturbation analysis. Working directly in fractional Sobolev spaces, we exploit asymmetric $L^2$--$L^\infty$ pairings through Kato-Ponce estimates, which place the highest-order derivatives on the approximate solutions. This substantially simplifies the perturbation analysis and yields a more direct justification of the $H^s$ approximation.

	\subsection{Structure of the paper and notation}
	The remainder of this paper is structured as follows. In Section \ref{sec-pre}, we collect the necessary preliminary tools, including the cylindrical coordinate formulations and fundamental estimates. In Section \ref{sec-approx}, we explicitly construct the initial data $(u_0,b_0)$ and define the approximate solutions $(\bar{u}, \bar{b})$. We subsequently compute the precise lower bounds of $\bar{b}$ at the critical time $t^*$ to capture the norm inflation. Section \ref{sec-pertur} is devoted to the perturbation analysis. Using a bootstrap argument, we prove that the difference between the exact solution and the approximate solution remains suitably controlled up to time $t^*$. In Section \ref{sec-proof}, we combine the bootstrap estimates to complete the proof of Theorem \ref{th-ill-ideal} for the ideal MHD equations. Finally, in Section \ref{sec-proof2}, we show that the same framework extends naturally to the non-resistive, non-viscous, and viscous and resistive MHD equations, thereby completing the proof of Theorem \ref{th-ill-mhd-unified}.
	\begin{notations}
		For $1\leq p\leq\infty$, we denote by $L^p(\R^d)$ the Lebesgue space equipped with the norm $\|\cdot\|_{L^p(\R^d)}$. For $s\in\R$, we denote by $H^s(\R^d)$ and $\dot H^s(\R^d)$ the inhomogeneous and homogeneous Sobolev spaces, respectively. The notation $a\lesssim b$ means that $a\leq Cb$ for some uniform constant $C$, which may differ from line to line. We write $a\lesssim_d b$ when the implicit constant depends on $d$. We write $a\sim b$ if both $a\lesssim b$ and $a\gtrsim b$ hold. Let $[\cdot,\cdot]$ denote the commutator of $A$ and $B$, i.e., $[A,B]=AB-BA$. We also use $\langle\cdot,\cdot\rangle$ to denote the standard inner product in $L^2(\R^d)$.
	\end{notations}

	\section{Preliminaries}\label{sec-pre}
	In this section, we collect several fundamental tools that will be frequently used in our analysis. Before introducing the analytical tools, we briefly recall the setup of cylindrical coordinates in $\R^3$, which forms the geometric foundation of our axisymmetric construction. 
	
	Let $(\theta,r, z)$ denote the standard cylindrical coordinates defined by
	\begin{equation*}
	x_1 = r \cos \theta, \quad x_2 = r \sin \theta, \quad x_3 = z,
	\end{equation*}
	with the associated orthonormal basis vectors given by
	\begin{equation*}
	\mathbf{e}_r = (\cos \theta, \sin \theta, 0), \quad \mathbf{e}_\theta = (-\sin \theta, \cos \theta, 0), \quad \mathbf{e}_z = (0, 0, 1).
	\end{equation*}
	For any vector field $v: \R^3 \to \R^3$, its cylindrical decomposition is expressed as $v = v_r \mathbf{e}_r + v_\theta \mathbf{e}_\theta + v_z \mathbf{e}_z$. We say a scalar function $f: \R^3 \to \R$ or a vector field $v$ is axisymmetric if all of its cylindrical components, namely $f$ or $v_r, v_\theta, v_z$, are independent of the angular variable $\theta$.
	
	Next, we collect the well-known Kato-Ponce commutator estimate and product estimate, which are essential for controlling the convective terms.
	\begin{lemma}{\rm\cite{kato1988commutator}}\label{lem:commutator}
		Assume that $s > 0$, $p \in (1, \infty)$, and $p_1, p_2, p_3, p_4 \in [1, \infty]$ such that $\frac{1}{p} = \frac{1}{p_1} + \frac{1}{p_2} = \frac{1}{p_3} + \frac{1}{p_4}$. Then there exists a constant $C > 0$ such that
		\begin{equation*}
		\|[\Lambda^s, f]g\|_{L^p} \le C \left( \|\Lambda^s f\|_{L^{p_1}} \|g\|_{L^{p_2}} + \|\nabla f\|_{L^{p_3}} \|\Lambda^{s-1} g\|_{L^{p_4}} \right).
		\end{equation*}
	\end{lemma}
	
	\begin{lemma}{\rm\cite{kato1988commutator}}\label{lem:fractional_leibniz}
		Assume that $s > 0$, $p \in (1, \infty)$, and $p_1, p_2, p_3, p_4 \in [1, \infty]$ such that $\frac{1}{p} = \frac{1}{p_1} + \frac{1}{p_2} = \frac{1}{p_3} + \frac{1}{p_4}$. Then there exists a constant $C > 0$ such that
		\begin{equation*}
		\|\Lambda^s (fg)\|_{L^p} \le C \left( \|f\|_{L^{p_1}} \|\Lambda^s g\|_{L^{p_2}} + \|g\|_{L^{p_3}} \|\Lambda^s f\|_{L^{p_4}} \right).
		\end{equation*}
	\end{lemma}
	
    The most crucial one is the classical \textit{a priori} energy estimate for the ideal MHD equations in Sobolev spaces, which relies on the above Lemmas.
    \begin{prop}\label{prop:a_priori}
    	Let $d \ge 2$ and $(u, b)$ be a smooth solution of the ideal MHD equations on $[0, t_0]$ for some $t_0 > 0$ such that
    	\begin{equation*}
    	\|\nabla u\|_{L^\infty([0, t_0]; L^\infty(\R^d))} + \|\nabla b\|_{L^\infty([0, t_0]; L^\infty(\R^d))} \le M
    	\end{equation*}
    	for some constant $M \ge 1$. Then for any $k \ge 0$ and $1 < p < \infty$, the following estimate holds for all $t \in [0, t_0]$
    	\begin{equation*}
    	\|u(t)\|_{W^{k,p}(\R^d)} + \|b(t)\|_{W^{k,p}(\R^d)} \le \left( \|u_0\|_{W^{k,p}(\R^d)} + \|b_0\|_{W^{k,p}(\R^d)} \right) \exp\left({C_{k,p} M t}\right),
    	\end{equation*}
    	where $C_{k,p} > 0$ is a universal constant independent of $M$, $t_0$, and $t$.
    \end{prop}
    We will also use the standard differential form of Gronwall's inequality.
    \begin{lemma}\label{lem:gronwall}
    	Let $\eta(t)$ be a nonnegative, absolutely continuous function on $[0, T]$ satisfying
    	\begin{equation*}
    	\frac{\d}{\d t}\eta(t) \le \alpha(t)\eta(t) + \beta(t)
    	\end{equation*}
    	for nonnegative, integrable functions $\alpha$ and $\beta$ on $[0, T]$. Then, for all $t \in [0, T]$,
    	\begin{equation*}
    	\eta(t) \le \exp\left(\int_0^t \alpha(\tau)\, \d\tau\right) \left( \eta(0) + \int_0^t \beta(\tau)\, \d\tau \right).
    	\end{equation*}
    \end{lemma}
    \section{The approximate solution}\label{sec-approx}
    In this section, we explicitly construct the initial data $(u_0,b_0)$ and the corresponding approximate solutions for the 3D MHD equations under the Magnetic-solo ansatz.
    
    \subsection{Construction of initial data}
    Our construction begins with a carefully designed anisotropic geometry for the initial data: a purely poloidal, swirl-free velocity field and a purely toroidal magnetic field. To capture the high-frequency oscillations and control the geometric support, we introduce two parameters $\mu\gg\nu\gg 1$, where $\mu$ represents the magnitude of spatial derivatives and $\nu^{-1}$ characterizes the distance from the support region from the symmetry axis.

    Let $\mu\gg 1$ be a large parameter. We set
    \begin{equation}\begin{cases}
    \gamma=\frac{\frac{5}{2}-s}{100}>0,\\
    \nu = \mu^{1-\gamma}.
    \end{cases}
    \end{equation}
    \begin{figure}[htbp]
    	\centering
    	\includegraphics[width=0.82\textwidth]{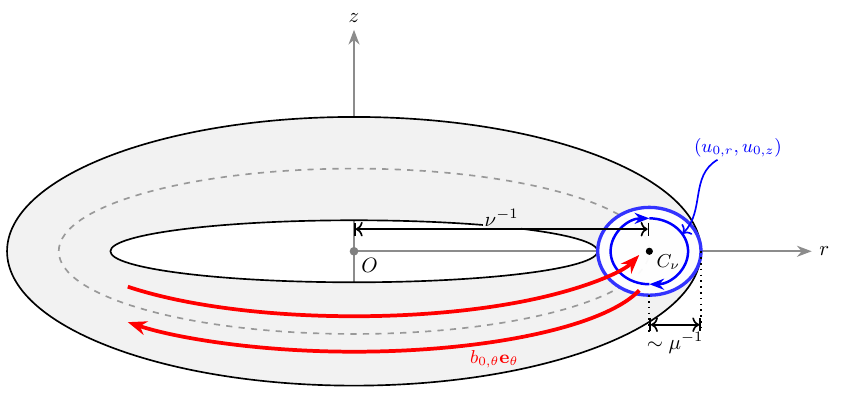}
    	\caption{3D schematic of the initial data construction. The velocity field $(u_{0,r}, u_{0,z})$ (blue) is purely poloidal, forming a steady 2D Euler vortex ring in the cross-section. The magnetic field $b_{0,\theta}\mathbf{e}_\theta$ (red) is purely toroidal, with opposite orientations in the upper and lower parts of the cross-section.}
    	\label{fig:3d_torus}
    \end{figure}
    Using the shifted polar coordinates $(\rho, \phi)$ centered at $(r, z) = (\nu^{-1}, 0)$ in the $rz$-plane, we  define
    \begin{equation}\label{u0-ideal}
    \begin{cases}
    u_{0,\theta}(z,r) =0 \\
    u_{0,r}(z,r) = -\varepsilon^2 \mu^{1 - s}\nu^{\frac{1}{2}}f_u'(\mu \rho)\partial_z\rho \\
    u_{0,z}(z,r) = \varepsilon^2 \mu^{1 - s}\nu^{\frac{1}{2}}f_u'(\mu \rho)\partial_r\rho \\
    \end{cases}
    \end{equation}
    and
    \begin{equation}\label{b0-ideal}
    \begin{cases}
    b_{0,\theta}(z,r) = \varepsilon^2 \mu^{1 - s}\nu^{\frac{1}{2}}g_b(\mu \rho)\sin (\phi) \\
    b_{0,r}(z,r) = 0 \\
    b_{0,z}(z,r) = 0
    \end{cases}
    \end{equation}
    where $f_u, g_b\in C^{\infty}_c(\R)$ with $g_b\not\equiv0$ such that
    \begin{equation}\label{suppfg}
    \begin{cases}
    \operatorname{supp} f_u \subset \left\{\frac{1}{2}\leq \rho \leq 2\right\} \\
    \operatorname{supp} g_b \subset \left\{1\leq \rho \leq \frac{3}{2}\right\}\\
    f_u'(\rho) = 1 \quad \text{for}~ 1\leq \rho \leq \frac{3}{2}.
    \end{cases}
    \end{equation}
	In our construction, we assume a purely poloidal velocity field and a purely azimuthal magnetic field, which implies $u_{0,\theta}=0$ and $b_{0,r}=b_{0,z}=0$. Under this geometric ansatz, the 2D divergence-free condition for the magnetic field in the $rz$-plane, $\partial_r b_{0,r} + \partial_z b_{0,z}=0$, is trivially satisfied.
	
	Consequently, the tuple $(u_{0,r},u_{0,z}, b_{0,r},b_{0,z})$ satisfies the 2D stationary ideal MHD equations, provided that the main poloidal velocity $(u_{0,r},u_{0,z})$ is a solution to the 2D stationary incompressible Euler equations
	\begin{equation}\label{2d_euler}
	\begin{cases}
	(u_{0,r}\partial_{r} + u_{0,z}\partial_{z})u_{0,r} + \partial_{r}p_{0} = 0 \\
	(u_{0,r}\partial_{r} + u_{0,z}\partial_{z})u_{0,z} + \partial_{z}p_{0} = 0 \\
	\partial_{r}u_{0,r} + \partial_{z}u_{0,z} = 0
	\end{cases}
	\quad \text{in } (r,z)\in \R^{2},
	\end{equation}
	with the scalar pressure
	\begin{equation}\label{pressure(r,z)}
	p_0(r,z)=-\varepsilon^4\mu^{2-2s}\nu
	\int_\rho^\infty
	\frac{\left[f_u'(\mu\bar\rho)\right]^2}{\bar\rho}\,\d\bar\rho.
	\end{equation}
	
	To ensure that $\operatorname{div} u_0 = 0$ in the 3D cylindrical coordinates, we introduce a divergence corrector for the velocity field
	\begin{equation*}
	u_c(r, z) := \varepsilon^2 \mu^{-s} \nu^{\frac{1}{2}} \frac{f_u(\mu\rho)}{r}.
	\end{equation*}
	Since the magnetic field $b_0$ is purely toroidal and axisymmetric, it naturally satisfies 
	\begin{equation}\label{divb0}
	 \operatorname{div} b_0 = \frac{1}{r}\partial_\theta b_{0,\theta} = 0,
	\end{equation}
	thus requiring no divergence corrector.
	
	Now, we define the initial data $u_0, b_0 : \R^3 \to \R^3$ for the ideal MHD equations in cylindrical coordinates by
	\begin{equation}\label{ini-data}
	\begin{cases}
	u_0 = u_{0,r} \mathbf{e}_r + (u_{0,z} + u_c) \mathbf{e}_z \\
	b_0 = b_{0,\theta} \mathbf{e}_\theta.
	\end{cases}
	\end{equation}
	
	\begin{remark}
	The constructed initial data $(u_0, b_0)$ are smooth, divergence-free vector fields, compactly supported in a thin solid torus at a distance $\sim \nu^{-1}$ from the origin, with cross-sectional diameter $\sim \mu^{-1}$. Crucially, the velocity field is swirl-free, while the magnetic field is purely toroidal.
	\end{remark}
	
	\subsection{Basic estimates of the initial data}
	In this subsection, we demonstrate that the initial data $(u_0, b_0)$ constructed above matches the required Sobolev regularity and smallness conditions.
	\begin{lemma}\label{lem:initial_estimates}
		The vector fields $u_0, b_0 : \R^3 \to \R^3$ defined by \eqref{ini-data} are smooth, divergence-free, and compactly supported in a thin solid torus of measure $\sim \mu^{-2}\nu^{-1}$. In addition, for any integer $k \ge 0$ and $1 \le p \le \infty$, there exists a constant $C_{k,p} > 0$ independent of $\varepsilon, \mu$ and $\nu$ such that
		\begin{equation}\label{eq:Wkp_bound}
		\|u_0\|_{W^{k,p}(\R^3)} + \|b_0\|_{W^{k,p}(\R^3)} \le C_{k,p} \varepsilon^2 \mu^{k-s} \mu^{1-\frac{2}{p}} \nu^{\frac{1}{2}-\frac{1}{p}}.
		\end{equation}
		In particular, for $0 < s <\frac{5}{2}$, there exists a universal constant $\varepsilon_0 \in (0, 1)$ such that for any $0 < \varepsilon \le \varepsilon_0$,
		\begin{equation}
		\|u_0\|_{H^s(\R^3)} + \|b_0\|_{H^s(\R^3)} \le \varepsilon.
		\end{equation}
	\end{lemma}

   \begin{proof}
   	We first verify the divergence-free condition. The zero divergence of $b_0$ is shown in \eqref{divb0}. For the velocity field, a direct computation verifies that $u_0$ is divergence‑free. Indeed,
   	\begin{align*}
   	\operatorname{div} u_0 
   	&= \partial_r u_{0,r} + \frac{u_{0,r}}{r} + \partial_z u_{0,z} + \partial_z u_c \\
   	&= \frac{u_{0,r}}{r} + \partial_z u_c \\
   	&= \varepsilon^2 \mu^{1-s} \nu^{\frac{1}{2}} \left( -\frac{f_u'(\mu\rho)\,\partial_z\rho}{r} + \frac{f_u'(\mu\rho)\,\partial_z\rho}{r} \right) = 0.
   	\end{align*}
   	Next, we establish the $W^{k,p}$ estimates. The support of the profile functions $f_u(\mu\rho)$ and $g_b(\mu\rho)$ is strictly confined within $\rho \sim \mu^{-1}$. Thus, the support $\mathcal{R}_\mu \subset \R^3$ is a thin solid torus centered at $r = \nu^{-1}$, and its Lebesgue measure is approximately $\text{Vol}(\mathcal{R}_\mu) \sim \mu^{-2}\nu^{-1}$.
   	
   	Noticing that each spatial derivative yields a factor of $\mu$, we obtain the point-wise bounds for the components $u_{0,r}, u_{0,z}$, and $b_{0,\theta}$
   	\begin{equation*}
   	|\nabla^k u_{0,r}| + |\nabla^k u_{0,z}| + |\nabla^k b_{0,\theta}| \lesssim \varepsilon^2 \mu^{k+1-s} \nu^{\frac{1}{2}}.
   	\end{equation*}
   	Integrating over the support $\mathcal{R}_\mu$ directly gives
   	\begin{align*}
   	\|\nabla^k u_{0,r}\|_{L^p} + \|\nabla^k u_{0,z}\|_{L^p} + \|\nabla^k b_{0,\theta}\|_{L^p} &\lesssim \varepsilon^2 \mu^{k+1-s} \nu^{\frac{1}{2}} (\mu^{-2}\nu^{-1})^{\frac{1}{p}} \\
   	&= \varepsilon^2 \mu^{k-s} \mu^{1-\frac{2}{p}} \nu^{\frac{1}{2}-\frac{1}{p}}.
   	\end{align*}
   	For the divergence corrector $u_c = \varepsilon^2 \mu^{-s} \nu^{\frac{1}{2}} \frac{f_u(\mu\rho)}{r}$, using $r \sim \nu^{-1}$ on the support, we similarly have
   	\begin{align*}
   	\|\nabla^k u_c\|_{L^p} &\lesssim \varepsilon^2 \mu^{k-s} \nu^{\frac{3}{2}} (\mu^{-2}\nu^{-1})^{\frac{1}{p}} \notag\\
   	&= \varepsilon^2 \mu^{k-s} \mu^{1-\frac{2}{p}} \nu^{\frac{1}{2}-\frac{1}{p}}\left(\mu^{-1}\nu\right).
   	\end{align*}
   	Since $\nu = \mu^{1-\gamma}$ and $\mu \gg 1$, comparing $u_c$ with the leading-order terms yields a ratio of $\mu^{-1}\nu = \mu^{-\gamma} \ll 1$. Therefore, the corrector $u_c$ is higher-order small and can be absorbed, which completes the proof of the $W^{k,p}$ bounds.
   	
   	Finally, taking $k = s$ (with interpolation for fractional $s$) and $p = 2$, the scaling factors in $\mu$ and $\nu$ cancel, yielding
   	\begin{equation*}
   	\|u_0\|_{H^s(\R^3)} + \|b_0\|_{H^s(\R^3)} \le C \varepsilon^2.
   	\end{equation*}
   	By choosing $\varepsilon_0 \in (0, 1)$ sufficiently small such that $C \varepsilon_0 \le 1$, we conclude that $\|u_0\|_{H^s} + \|b_0\|_{H^s} \le \varepsilon$ holds for any $0 < \varepsilon \le \varepsilon_0$. 
   \end{proof}
	
	\subsection{The approximate solution}
	Due to the specific decoupled geometry of the initial data $(u_0, b_0)$, we approximate the exact solution of the ideal MHD equations by keeping the poloidal velocity stationary and transporting the purely toroidal magnetic field.   We refer to this approximate evolution as the \textbf{Magnetic-solo ansatz}.
	
	Under this ansatz, the approximate solution $(\bar{u}, \bar{b}) : \R^+ \times \R^3 \to \R^3 \times \R^3$ is given by
	\begin{equation}\label{approx_sol}
	\begin{aligned}
	\bar{u}(t, x) & = u_0(x)= u_{0,r} \mathbf{e}_r + (u_{0,z} + u_c) \mathbf{e}_z, \\
	\bar{b}(t, x) &= \bar{b}_\theta \mathbf{e}_\theta,
	\end{aligned}
	\end{equation}
	where $\bar{u}$ is independent of time, and the toroidal component $\bar{b}_\theta$ is the unique smooth solution to the linear transport equation driven by the leading-order poloidal velocity, satisfying  
	\begin{equation}\label{transport_b}
	\begin{cases}
	\partial_t \bar{b}_\theta + (u_{0,r} \partial_r + u_{0,z} \partial_z)\bar{b}_\theta = 0,\\ 
    \bar{b}_{\theta}|_{t=0} = b_{0,\theta},
	\end{cases}
	\quad \text{for}~ (t,r,z)\in \R^+ \times \R^2.
	\end{equation}
	In the shifted polar coordinates $(\rho, \phi)$, recall the chain rule $\partial_\phi = -\rho \sin(\phi) \partial_r + \rho \cos(\phi) \partial_z$. We can rewrite the convective operator by \eqref{u0-ideal}
	\begin{align*}
	u_{0,r} \partial_r + u_{0,z} \partial_z &= \left(-\varepsilon^2 \mu^{1-s} \nu^{\frac{1}{2}} f_u'(\mu\rho) \sin(\phi) \right) \partial_r + \left(\varepsilon^2 \mu^{1-s} \nu^{\frac{1}{2}} f_u'(\mu\rho) \cos(\phi) \right) \partial_z \\
	&= \frac{\varepsilon^2 \mu^{1-s} \nu^{\frac{1}{2}} f_u'(\mu\rho)}{\rho} \left( -\rho \sin(\phi) \partial_r + \rho \cos(\phi) \partial_z \right) \\
	&= \frac{\varepsilon^2 \mu^{1-s} \nu^{\frac{1}{2}} f_u'(\mu\rho)}{\rho} \partial_\phi.
	\end{align*}
	Since $f_u'(\mu\rho) = 1$ on the support of $g_b(\mu\rho)$, the linear transport equation \eqref{transport_b} simplifies to
	\begin{equation*}
	\partial_t \bar{b}_\theta + \frac{\varepsilon^2 \mu^{1-s}\nu^{1/2}}{\rho} \partial_\phi \bar{b}_\theta = 0. 
	\end{equation*}
	Solving this equation with the initial data $\bar{b}_\theta|_{t=0} = b_{0,\theta}$ in \eqref{b0-ideal} yields
	\begin{equation}\label{explicit_b}
	\bar{b}_\theta(t, r, z) = \varepsilon^2 \mu^{1-s} \nu^{\frac{1}{2}} g_b(\mu\rho) \sin\left( \phi - t \frac{\varepsilon^2 \mu^{1-s} \nu^{\frac{1}{2}}}{\rho} \right).
	\end{equation}
	
	To capture the norm inflation, we define the critical time $t^*$ as
	\begin{equation}\label{t*}
	t^* = \varepsilon^{-N-2} \mu^{-2+s} \nu^{-\frac{1}{2}},
	\end{equation}
	where $N = \max\left\{\frac{10}{s}, 100\right\}$. Notice that $t^* \to 0$ as $\mu \to \infty$.
	
	\begin{prop}\label{prop:approximation}
		The vector fields $\left(\bar{u}, \bar{b}\right)$ are smooth, divergence-free, and satisfy the following estimates for any $k \ge 0$ and $1 \le p \le \infty$ on the interval $t \in [0, t^*]$,
		\begin{equation}\label{eq:approx_Wkp}
		\begin{aligned}
		\|\bar{u}(t)\|_{W^{k,p}(\R^3)} &\le C_{k,p} \varepsilon^2 \mu^{k-s} \mu^{1-\frac{2}{p}} \nu^{\frac{1}{2}-\frac{1}{p}}, \\
		\|\bar{b}(t)\|_{W^{k,p}(\R^3)} &\le C_{k,p} \varepsilon^{2-kN} \mu^{k-s} \mu^{1-\frac{2}{p}} \nu^{\frac{1}{2}-\frac{1}{p}}.
		\end{aligned}
		\end{equation}
		Moreover, $\left(\bar{u}, \bar{b}\right)$ is an approximate solution to the ideal MHD equations. There exist a smooth pressure $\bar{P} = p_0(r,z)$ and error fields $\overline{E}_u, \overline{E}_b$ such that
		\begin{equation}\label{error_eqs}
		\begin{cases}
		\partial_t \bar{u} + \bar{u} \cdot \nabla \bar{u} - \bar{b} \cdot \nabla \bar{b} + \nabla \bar{P} = \overline{E}_u, \\
		\partial_t \bar{b} + \bar{u} \cdot \nabla \bar{b} - \bar{b} \cdot \nabla \bar{u} = \overline{E}_b,\\
		(\bar{u},\bar{b})|_{t = 0} = (u_{0}, b_{0}),
		\end{cases}
		\end{equation}
		where they satisfy the uniform $L^2$ bound for $t \in [0, t^*]$,
		\begin{equation}\label{error_L2}
		\|\nabla^k\overline{E}_u(t)\|_{L^2(\R^3)}+\|\nabla^k\overline{E}_b(t)\|_{L^2(\R^3)} \le C_{\varepsilon,k}\mu^{k-s}\left(\mu^{-1}\nu\right)\mu^{2-s} \nu^{\frac{1}{2}}.
		\end{equation}
	\end{prop}

    \begin{proof}
    \noindent \textbf{Part 1: Estimates of $\bar{u}$ and $\bar{b}$} \\
    The $W^{k,p}$ estimates for $\bar{u}=u_{0,r} \mathbf{e}_r + (u_{0,z} + u_c) \mathbf{e}_z$ trivially follow from Lemma \ref{lem:initial_estimates} since $\bar{u}$ is stationary.
    
    For the toroidal magnetic field, the explicit formula is given by \eqref{explicit_b}. By the chain rule, each spatial differentiation yields an amplification factor bounded by \[\max\left\{\mu, t \left| \nabla ( \rho^{-1} \varepsilon^2 \mu^{1-s} \nu^{\frac{1}{2}} ) \right|\right\}.\] 
    On the support where $\rho \sim \mu^{-1}$, we have $\nabla(\rho^{-1}) \sim \rho^{-2} \sim \mu^2$. Hence, a spatial derivative yields a maximum factor bounded by
    \begin{equation*}
    \max\left\{\mu, t \varepsilon^2 \mu^{3-s} \nu^{\frac{1}{2}} \right\}.
    \end{equation*}
    For $t \le t^*$ with $t^*$ defined in \eqref{t*}, we have $t \varepsilon^2 \mu^{3-s} \nu^{\frac{1}{2}} \le \varepsilon^{-N} \mu$. Thus, the maximum amplification factor per differentiation in \eqref{explicit_b} is $\varepsilon^{-N} \mu$. Multiplying the amplitude and the $L^p$ measure factor $\mu^{-\frac{2}{p}}\nu^{-\frac{1}{p}}$, we obtain
    \begin{align*}
    \|\nabla^k \bar{b}(t)\|_{L^p(\R^3)} &\le C_k \left(\varepsilon^2 \mu^{1-s} \nu^{\frac{1}{2}}\right) \left(\varepsilon^{-N}\mu\right)^k\left(\mu^{-\frac{2}{p}}\nu^{-\frac{1}{p}}\right) \notag\\
    &= C_k \varepsilon^{2-kN} \mu^{k-s}\mu^{1-\frac{2}{p}} \nu^{\frac{1}{2}-\frac{1}{p}}.
    \end{align*}
    
    \noindent \textbf{Part 2: Derivation and estimates of the error fields} \\
    In cylindrical coordinates $(r, \theta, z)$, assuming axisymmetry, \eqref{error_eqs} can be rewritten as
    \begin{equation}\label{axisymmetric_mhd}
    \begin{cases}
    \partial_t \bar{u}_\theta + \bar{u}_r \partial_r \bar{u}_\theta + \bar{u}_z \partial_z \bar{u}_\theta + \frac{\bar{u}_r \bar{u}_\theta}{r} - \left( \bar{b}_r \partial_r \bar{b}_\theta + \bar{b}_z \partial_z \bar{b}_\theta + \frac{\bar{b}_r \bar{b}_\theta}{r} \right) = \overline{E}_{u,\theta}, \\[8pt]
    \partial_t \bar{u}_r + \bar{u}_r \partial_r \bar{u}_r + \bar{u}_z \partial_z \bar{u}_r - \frac{\bar{u}_\theta^2}{r} + \partial_r \bar{P} - \left(\bar{b}_r \partial_r \bar{b}_r + \bar{b}_z \partial_z \bar{b}_r - \frac{\bar{b}_\theta^2}{r} \right) = \overline{E}_{u,r}, \\[8pt]
    \partial_t \bar{u}_z + \bar{u}_r \partial_r \bar{u}_z + \bar{u}_z \partial_z \bar{u}_z + \partial_z \bar{P} - \left(\bar{b}_r \partial_r \bar{b}_z + \bar{b}_z \partial_z \bar{b}_z \right)= \overline{E}_{u,z}, \\[8pt]
    \partial_t \bar{b}_\theta + \bar{u}_r \partial_r \bar{b}_\theta + \bar{u}_z \partial_z \bar{b}_\theta + \frac{\bar{b}_r \bar{u}_\theta}{r} - \left( \bar{b}_r \partial_r \bar{u}_\theta + \bar{b}_z \partial_z \bar{u}_\theta + \frac{\bar{u}_r \bar{b}_\theta}{r} \right) = \overline{E}_{b,\theta}, \\[8pt]
    \partial_t \bar{b}_r + \bar{u}_r \partial_r \bar{b}_r + \bar{u}_z \partial_z \bar{b}_r - \left( \bar{b}_r \partial_r \bar{u}_r + \bar{b}_z \partial_z \bar{u}_r \right) = \overline{E}_{b,r}, \\[8pt]
    
    \partial_t \bar{b}_z + \bar{u}_r \partial_r \bar{b}_z + \bar{u}_z \partial_z \bar{b}_z - \left( \bar{b}_r \partial_r \bar{u}_z + \bar{b}_z \partial_z \bar{u}_z \right) = \overline{E}_{b,z}.
    \end{cases}
    \end{equation}
    We substitute the approximate solutions $\bar{u}_{\theta}=0$, $\bar{u}_r= u_{0,r}$, $\bar{u}_z = u_{0,z} + u_c$, $\bar{b}_r = \bar{b}_z = 0$, and the pressure $\bar{P} = p_0(r,z)$ from \eqref{pressure(r,z)} into the system \eqref{axisymmetric_mhd}. Recall that the leading-order terms satisfy the 2D stationary Euler equations \eqref{2d_euler} and $\bar{b}_{\theta}$ satisfies the transport equation \eqref{transport_b}, we obtain the nonzero components of the error fields
    \begin{equation*}
    \begin{cases}
    \overline{E}_{u,r} = u_c \partial_z u_{0,r} + \frac{\bar{b}_\theta^2}{r}, \\[8pt]
    \overline{E}_{u,z} = u_{0,r} \partial_r u_c + u_{0,z} \partial_z u_c + u_c \partial_z u_{0,z} + u_c \partial_z u_c, \\[8pt]
    \overline{E}_{b,\theta} = u_c \partial_z \bar{b}_\theta - \frac{u_{0,r} \bar{b}_\theta}{r}.
    \end{cases}
    \end{equation*}
    By differentiating the expressions for the error fields, we
    obtain an additional factor $\mu^k$, up to a constant depending on
    $\varepsilon$ and $k$. Hence, it suffices to focus on the corresponding $L^2$ estimates.
    
    Using the estimates from Lemma \ref{lem:initial_estimates} and the relation $\frac{1}{r} \sim \nu$, the following error estimates are obtained. For the radial error $\overline{E}_{u,r}$, we have
    \begin{align}\label{error-r}
    \|\overline{E}_{u,r}\|_{L^2(\R^3)} &\le \|u_c\|_{L^2} \|\partial_z u_{0,r}\|_{L^{\infty}} + \left\| \frac{1}{r} \right\|_{L^\infty} \|\bar{b}_\theta\|_{L^\infty} \|\bar{b}_\theta\|_{L^2} \notag\\
    &\lesssim \left(\varepsilon^2 \mu^{-1-s}\nu\right) \cdot \left(\varepsilon^2 \mu^{2-s}\nu^{\frac{1}{2}}\right)+\nu \cdot \left(\varepsilon^2 \mu^{1-s} \nu^{\frac{1}{2}}\right) \cdot \left(\varepsilon^2 \mu^{-s}\right) \notag\\
    &\lesssim \varepsilon^4 \mu^{1-2s} \nu^{\frac{3}{2}}.
    \end{align}
    For the axial error $\overline{E}_{u,z}$, we get
    \begin{align}\label{error-z}
    \|\overline{E}_{u,z}\|_{L^2(\R^3)}&\le \|u_{0,r}\|_{L^\infty} \|\partial_r u_c\|_{L^2}+\|u_{0,z}\|_{L^\infty} \|\partial_z u_c\|_{L^2}+\|u_c\|_{L^\infty} \|\partial_z u_{0,z}\|_{L^2}+\|u_c\|_{L^\infty} \|\partial_z u_c\|_{L^2} \notag\\
    &\lesssim \varepsilon^4 \mu^{1-2s} \nu^{\frac{3}{2}}+ \varepsilon^4 \mu^{-2s} \nu^{\frac{5}{2}}\notag\\
    &\lesssim \varepsilon^4 \mu^{1-2s} \nu^{\frac{3}{2}}.
    \end{align}
    For the toroidal magnetic error $\overline{E}_{b,\theta}$, we similarly obtain
    \begin{align}\label{error-theta}
    \|\overline{E}_{b,\theta}\|_{L^2(\R^3)} &\le \|u_c\|_{L^2} \|\partial_z \bar{b}_\theta\|_{L^\infty} + \left\| \frac{1}{r} \right\|_{L^\infty} \|u_{0,r}\|_{L^\infty} \|\bar{b}_\theta\|_{L^2} \notag\\
    &\lesssim_{\varepsilon} \left(\mu^{-s-1}\nu\right) \cdot \left(\mu^{2-s}\nu^{\frac{1}{2}}\right)+\nu\cdot\left(\mu^{1-s}\nu^{\frac{1}{2}}\right) \cdot \mu^{-s} \notag\\
    &\lesssim_{\varepsilon} \mu^{1-2s} \nu^{\frac{3}{2}}.
    \end{align}
    Summing the estimates \eqref{error-r}, \eqref{error-z}, and \eqref{error-theta}, we obtain the desired uniform $L^2$ bound for the total error field
    \begin{equation*}
    \|\overline{E}_u(t)\|_{L^2(\R^3)} + \|\overline{E}_b(t)\|_{L^2(\R^3)} \le C_\varepsilon \mu^{1-2s} \nu^{\frac{3}{2}},
    \end{equation*}
    which completes the proof of Proposition \ref{prop:approximation}.
    \end{proof}
	
	\begin{lemma}\label{lem:norm_inflation}
		There exists a universal constant $\varepsilon_0 \in (0, 1)$ such that for any $0 < \varepsilon \le \varepsilon_0$, there holds the lower bound
		\begin{equation*}
		\|\bar{b}(t^*)\|_{H^s(\R^3)} \ge \varepsilon^{-2}
		\end{equation*}
		where $t^*>0$ is the critical time defined in \eqref{t*}.
	\end{lemma}
	\begin{proof}
		We first establish a lower bound for the $\dot{H}^1$ norm of $\bar{b}$ at $t = t^*$. By definition, the approximate magnetic field is purely toroidal, $\bar{b} = \bar{b}_\theta \mathbf{e}_\theta$. Evaluating the gradient tensor in cylindrical coordinates, the pointwise orthogonality yields $|\nabla \bar{b}|^2 = |\nabla \bar{b}_\theta|^2 + r^{-2}|\bar{b}_\theta|^2 \ge |\partial_\rho \bar{b}_\theta|^2$. Hence, it suffices to estimate the lower bound of $\|\partial_\rho \bar{b}_\theta(t^*)\|_{L^2(\R^3)}$.
		
		To simplify the notation, we denote the amplitude $A = \varepsilon^2 \mu^{1-s} \nu^{\frac{1}{2}}$. Recalling that the explicit formula for $\bar{b}_\theta$ is strictly given by $\bar{b}_\theta = A g_b(\mu\rho) \sin(\phi - t A \rho^{-1})$ in \eqref{explicit_b}. Taking the derivative with respect to the local polar radius $\rho$, we obtain
		\begin{equation}\label{b_rho}
		\partial_\rho \bar{b}_\theta = A \mu g_b'(\mu\rho) \sin(\phi - t A \rho^{-1}) + A g_b(\mu\rho) \cos(\phi - t A \rho^{-1}) \left( t A \rho^{-2} \right).
		\end{equation}
		For the term in parentheses, using $\rho \sim \mu^{-1}$ on $\mathcal{R}_\mu$, there exists $c_1>0$ such that $\rho^{-2} \ge c_1 \mu^2$. At the critical time $t = t^*$, the amplification factor is bounded from below by
		\begin{equation}\label{factor}
		t^* A \rho^{-2} \ge c_1 t^* A \mu^2 = c_1 \left(\varepsilon^{-N-2} \mu^{-2+s} \nu^{-\frac{1}{2}}\right) \left(\varepsilon^2 \mu^{1-s} \nu^{\frac{1}{2}}\right) \mu^2 = c_1 \varepsilon^{-N} \mu.
		\end{equation}
		Combining \eqref{b_rho} and \eqref{factor} with the triangle inequality, we extract the leading term
		\begin{align*}
		\|\partial_\rho \bar{b}_\theta(t^*)\|_{L^2} &\ge \left\| A g_b(\mu\rho) \cos(\phi - t^* A \rho^{-1}) \left( t^* A \rho^{-2} \right) \right\|_{L^2} - \left\| A \mu g_b'(\mu\rho) \sin(\phi - t^* A \rho^{-1}) \right\|_{L^2} \\[8pt]
		&\ge c_1 \left(\varepsilon^{-N} \mu\right) A \left\| g_b\cos(\phi - t^* A \rho^{-1}) \right\|_{L^2} - \mu A \left\| g_b' \sin(\phi - t^* A \rho^{-1}) \right\|_{L^2}.
		\end{align*}
		Integrating over the support $\mathcal{R}_\mu$, the $L^2$ measure factor scales as $\mu^{-1}\nu^{-\frac{1}{2}}$. Hence, we have
		\begin{equation*}
		\|\partial_\rho \bar{b}_\theta(t^*)\|_{L^2} \ge C_1 \varepsilon^{2-N} \mu^{1-s} - C_2 \varepsilon^2 \mu^{1-s}.
		\end{equation*}
		For any sufficiently small $\varepsilon \in (0, \varepsilon_0)$ and $N \ge 1$, the first term strictly dominates. Thus, we obtain the $\dot{H}^1$ lower bound
		\begin{equation}\label{eq:H1_lower_bound}
		\|\bar{b}(t^*)\|_{\dot{H}^1} \ge C \varepsilon^{2-N} \mu^{1-s}.
		\end{equation}
		Next, we deduce the $\dot{H}^s$ estimate by using Sobolev interpolation. We split the argument into two cases depending on $s$.
		
		\noindent \textbf{Case 1: $s \ge 1$.} By the standard interpolation inequality,
		\begin{equation*}
		\|\bar{b}\|_{\dot{H}^1} \le C \|\bar{b}\|_{L^2}^{1-\frac{1}{s}} \|\bar{b}\|_{\dot{H}^s}^{\frac{1}{s}}.
		\end{equation*}
		Rearranging the above inequality and utilizing \eqref{eq:approx_Wkp} in Proposition \ref{prop:approximation}, we obtain
		\begin{equation*}
		\|\bar{b}(t^*)\|_{\dot{H}^s} \ge C \frac{\|\bar{b}\|_{\dot{H}^1}^s}{\|\bar{b}\|_{L^2}^{s-1}} \ge C \frac{\left(\varepsilon^{2-N} \mu^{1-s}\right)^s}{\left(\varepsilon^2 \mu^{-s}\right)^{s-1}} = C \varepsilon^{2-Ns}.
		\end{equation*}
		\noindent \textbf{Case 2: $0 < s < 1$.} We interpolate $\dot{H}^1$ between $\dot{H}^s$ and $\dot{H}^2$:
		\begin{equation*}
		\|\bar{b}\|_{\dot{H}^1} \le C \|\bar{b}\|_{\dot{H}^s}^{\frac{1}{2-s}} \|\bar{b}\|_{\dot{H}^2}^{\frac{1-s}{2-s}}.
		\end{equation*}
		Rearranging the inequality and utilizing \eqref{eq:approx_Wkp} again, we deduce that
		\begin{equation*}
		\|\bar{b}(t^*)\|_{\dot{H}^s} \ge C \frac{\|\bar{b}\|_{\dot{H}^1}^{2-s}}{\|\bar{b}\|_{\dot{H}^2}^{1-s}} \ge C \frac{\left(\varepsilon^{2-N} \mu^{1-s}\right)^{2-s}}{\left(\varepsilon^{2-2N} \mu^{2-s}\right)^{1-s}}=C \varepsilon^{2-Ns}.
		\end{equation*}
		In either case, since $N$ is chosen such that $N \ge \frac{10}{s}$, we have $2-Ns\le -8$. Consequently, for any sufficiently small $0 < \varepsilon \le \varepsilon_0$, the homogeneous lower bound $\|\bar{b}(t^*)\|_{\dot{H}^s(\R^3)} \ge C\varepsilon^{-8}$ holds. 
		
		Recalling the standard relation between homogeneous and inhomogeneous Sobolev norms, we have $\|\bar{b}(t^*)\|_{H^s(\R^3)} \ge \|\bar{b}(t^*)\|_{\dot{H}^s(\R^3)}$. Therefore, taking $\varepsilon$ sufficiently small yields the desired norm inflation in the inhomogeneous space
		\begin{equation*}
		\|\bar{b}(t^*)\|_{H^s(\R^3)} \ge \varepsilon^{-2},
		\end{equation*}
		which completes the proof of Lemma \ref{lem:norm_inflation}.
	\end{proof}
	
	\section{Perturbation analysis}\label{sec-pertur}
	In this section, we prove that the approximate solution $(\bar{u}, \bar{b})$ constructed in Section \ref{sec-approx} remains close to the exact solution $(u, b)$ of the ideal MHD equations up to the critical time $t^*$, at which $\dot{H}^s$ norm inflation for the magnetic field occurs.
	
	Let $u$ and $b$ be the exact solutions to the ideal MHD equations with the initial data $(u_0, b_0)$. We define the perturbation variables as the differences between the exact solutions and the approximate solutions
	\begin{equation*}
	w = u- \bar{u}, \quad \beta= b - \bar{b}, \quad q = p-\bar{p}.
	\end{equation*}
	Substituting $u = \bar{u} + w$ and $b = \bar{b} + \beta$ into the ideal MHD equations and subtracting the equations for the approximate solutions yields the evolution equations for $(w,\beta)$
	\begin{equation}\label{eq:perturbation_system}
	\begin{cases}
	\partial_t w + u \cdot \nabla w +w \cdot \nabla \bar{u} + \nabla q = b \cdot \nabla \beta+ \beta \cdot \nabla \bar{b} - \overline{E}_u, \\
	\partial_t \beta + u \cdot \nabla \beta + w \cdot \nabla \bar{b}= b \cdot \nabla w + \beta \cdot \nabla \bar{u} - \overline{E}_b, \\
	\operatorname{div}w=0, \quad \operatorname{div}\beta=0, \\
	(w, \beta)|_{t=0} = (0, 0).
	\end{cases}
	\end{equation}
	The primary goal of this section is to prove the following proposition, which guarantees the validity of the approximation and excludes finite-time blow-up before $t^*$.
	\begin{prop}\label{prop:main_perturbation}
		Let $T = T(u_0, b_0) > 0$ be the maximal time of existence for the local-in-time smooth solution $(u, b)$ of the ideal MHD equations with the initial data defined by \eqref{ini-data}. For any $\varepsilon > 0$, there exists $\mu_0 > 0$ sufficiently large such that if $\mu \ge \mu_0$, then $0 < t^* < T$, namely, $(u, b) \in C([0,t^*];H^\infty(\R^3)\times H^\infty(\R^3))$.
		
		More quantitatively, for any $\varepsilon > 0$, if $\mu \ge \mu_0$, then
		\begin{equation}\label{w+beta-ineq}
		\|w\|_{L^\infty([0,t^*]; H^s)}+\| \beta\|_{L^\infty([0,t^*]; H^s)}\le C_{s,\varepsilon}\mu^{-\frac{\gamma}{2}},
		\end{equation}
		where $C_{s,\varepsilon}$ is independent of $\mu$.
	\end{prop}
	
	\begin{proof}
	We first define the energy functional for all $m \ge 0$ as	
    \begin{equation*}
    Y_m(t)=\left(\|w(t)\|_{H^m(\R^3)}^2+\|\beta(t)\|_{H^m(\R^3)}^2\right)^{\frac12}.
    \end{equation*}
    
    Before establishing the bootstrap argument, let $\varepsilon>0$ be fixed. By \eqref{eq:approx_Wkp} in Proposition \ref{prop:approximation}, there exists a constant $M_\varepsilon\ge1$, such that the approximate solutions satisfy
    \begin{equation}\label{eq:approx_grad_bound}
    \|\nabla \bar{u}(t)\|_{L^\infty(\R^3)} + \|\nabla \bar{b}(t)\|_{L^\infty(\R^3)} \le M_\varepsilon \mu^{2-s}\nu^{\frac{1}{2}}, \quad \text{for all } t \in [0, t^*].
    \end{equation}
    
    \noindent\textbf{Bootstrap assumption:} Let $t_0\in(0,t^*]$ be the maximal time such that
    \begin{equation}\label{bootstrap_assumption}
    \|\nabla w(t)\|_{L^\infty(\R^3)} + \|\nabla \beta(t)\|_{L^\infty(\R^3)} \le 2M_\varepsilon \mu^{2-s}\nu^{\frac{1}{2}}, \quad \text{for all } t \in [0, t_0].
    \end{equation}
  
    \noindent \textbf{Step 1: Energy estimates under the bootstrap assumption}\\
    To proceed with the proof of Proposition \ref{prop:main_perturbation}, we first formulate the energy estimates as a lemma.
    \begin{lemma}\label{lem:energy_est}
    	For any real number $k > \frac{5}{2}$, under the bootstrap assumption \eqref{bootstrap_assumption}, there exist constants $C_\varepsilon, C_{k,\varepsilon} > 0$ independent of $\mu$ and $\nu$ such that for all $t \in [0, t_0]$, we have
    	\begin{align}
    	Y_0(t) &\le C_\varepsilon \mu^{-s}(\mu^{-1}\nu), \label{eq:Y0_est}\\
    	Y_k(t) &\le C_{k,\varepsilon} \mu^{k-s}. \label{eq:Yk_est}
    	\end{align}
    \end{lemma}
    \begin{proof}
    	We first establish the $L^2$ estimate \eqref{eq:Y0_est}. Taking the $L^2$ inner product of the first and second equations in \eqref{eq:perturbation_system} with $w$ and $\beta$ respectively, and summing them, we obtain the energy identity. Notice that the convection terms vanish due to the divergence-free condition, namely $\langle u \cdot \nabla w, w \rangle = 0$, $\langle u \cdot \nabla \beta, \beta \rangle = 0$, and $\langle \nabla q, w \rangle = 0$. Moreover, we have $\langle b \cdot \nabla \beta, w \rangle =-\langle b \cdot \nabla w, \beta \rangle$.
    	
    	Thus, the $L^2$ energy inequality becomes
    	\begin{align*}
    	\frac{1}{2} \frac{\d}{\d t} Y_0^2(t) &= -\langle w \cdot \nabla \bar{u}, w \rangle + \langle \beta \cdot \nabla \bar{b}, w \rangle - \langle w \cdot \nabla \bar{b}, \beta \rangle + \langle \beta \cdot \nabla \bar{u}, \beta \rangle - \langle \overline{E}_u, w \rangle - \langle \overline{E}_b, \beta \rangle \\
   		&\le \left(\|\nabla \bar{u}\|_{L^\infty} + \|\nabla \bar{b}\|_{L^\infty} \right) Y_0^2(t)+\left(\|\overline{E}_u\|_{L^2} + \|\overline{E}_b\|_{L^2} \right) Y_0(t).
    	\end{align*}
    	Applying Gronwall's inequality over $[0, t_0]$ and $(w, \beta)|_{t=0} = (0, 0)$, we obtain
    	\begin{align*}
    	Y_0(t) &\lesssim \exp\left( \int_0^t \left(\|\nabla \bar{u}\|_{L^\infty} + \|\nabla \bar{b}\|_{L^\infty} \right) \mathrm{d}\tau \right)  \int_0^t \left(\|\overline{E}_u\|_{L^2} + \|\overline{E}_b\|_{L^2} \right) \mathrm{d}\tau\\
    	&\lesssim \exp\left(t^*  M_\varepsilon \mu^{2-s}\nu^{\frac{1}{2}} \right) \cdot t^* \left(\mu^{1-2s}\nu^{\frac{3}{2}}\right).
    	\end{align*}
    	Substituting $t^* = \varepsilon^{-N-2}\mu^{-2+s}\nu^{-\frac{1}{2}}$ in \eqref{t*}, we obtain $Y_0(t) \le C_\varepsilon \mu^{-s}(\mu^{-1}\nu)$ for any $t \in [0, t_0]$.
    	
    	Next, we evaluate the $H^k$ estimate \eqref{eq:Yk_est}. Acting $\Lambda^k$ to \eqref{eq:perturbation_system} and taking the $L^2$ inner products with $\Lambda^k w$ and $\Lambda^k \beta$ respectively, the pressure term vanishes. We isolate the highest-order derivatives in the nonlinear terms using commutators
    	\begin{align*}
    	\langle \Lambda^k(u \cdot \nabla w), \Lambda^k w \rangle &= \langle [\Lambda^k, u \cdot \nabla]w, \Lambda^k w \rangle + \langle u \cdot \nabla \Lambda^k w, \Lambda^k w \rangle = \langle [\Lambda^k, u \cdot \nabla]w, \Lambda^k w \rangle,
    	\end{align*}
    	where the term $\langle u \cdot \nabla \Lambda^k w, \Lambda^k w \rangle=0$ by $\operatorname{div} u=0$. Similarly, we have $\langle \Lambda^k(u \cdot \nabla \beta), \Lambda^k \beta \rangle= \langle [\Lambda^k, u \cdot \nabla]\beta, \Lambda^k \beta \rangle$. Moreover, we know
    	\begin{align*}
    	&\langle \Lambda^k(b \cdot \nabla \beta), \Lambda^k w \rangle + \langle \Lambda^k(b \cdot \nabla w), \Lambda^k \beta \rangle \\
    	&\quad= \langle [\Lambda^k, b \cdot \nabla]\beta, \Lambda^k w \rangle + \langle [\Lambda^k, b \cdot \nabla]w, \Lambda^k \beta \rangle + \langle b \cdot \nabla(\Lambda^k \beta \cdot \Lambda^k w), 1 \rangle,
    	\end{align*}
    	where the term $\langle b \cdot \nabla(\Lambda^k \beta \cdot \Lambda^k w), 1 \rangle$ vanishes by $\operatorname{div} b=0$. 
    	
    	Therefore, adding up the above identities, we deduce that
    	\begin{align*}
    	\frac{1}{2}\frac{\mathrm{d}}{\mathrm{d}t} Y_k^2(t) \le \; & \left( \|[\Lambda^k, u \cdot \nabla]w\|_{L^2} + \|[\Lambda^k, b \cdot \nabla]\beta\|_{L^2} + \|\Lambda^k(w \cdot \nabla \bar{u})\|_{L^2} + \|\Lambda^k(\beta \cdot \nabla \bar{b})\|_{L^2} \right) \|w\|_{H^k} \\
    	+ \; & \left(\|[\Lambda^k, u \cdot \nabla]\beta\|_{L^2} + \|[\Lambda^k, b \cdot \nabla]w\|_{L^2} + \|\Lambda^k(w \cdot \nabla \bar{b})\|_{L^2} + \|\Lambda^k(\beta \cdot \nabla \bar{u})\|_{L^2} \right) \|\beta\|_{H^k} \\
    	+ \; & \left( \|\overline{E}_u\|_{H^k} + \|\overline{E}_b\|_{H^k}\right) Y_k(t).
    	\end{align*}
    	We first use Lemma \ref{lem:commutator} to estimate the commutator terms. Substituting $u = \bar{u} + w$ and $b = \bar{b} + \beta$, then we obtain
    	\begin{align*}
    	\|[\Lambda^k, u \cdot \nabla]w\|_{L^2} &\lesssim \|\nabla u\|_{L^\infty} \|w\|_{H^k} + \|\nabla w\|_{L^\infty} \|\Lambda^k u\|_{L^2} \\
    	&\lesssim \left(\|\nabla \bar{u}\|_{L^\infty} + \|\nabla w\|_{L^\infty} \right) \|w\|_{H^k} + \|\nabla w\|_{L^\infty} \|\bar{u}\|_{H^k}, \\
    	\|[\Lambda^k, u \cdot \nabla]\beta\|_{L^2} 
    	&\lesssim \|\nabla u\|_{L^\infty} \|\beta\|_{H^k} + \|\nabla \beta\|_{L^\infty} \|\Lambda^k u\|_{L^2} \\
    	&\lesssim \left( \|\nabla \bar{u}\|_{L^\infty} + \|\nabla w\|_{L^\infty} \right) \|\beta\|_{H^k} + \|\nabla \beta\|_{L^\infty} \left(\|\bar{u}\|_{H^k}+\|w\|_{H^k}\right), \\
    	\|[\Lambda^k, b \cdot \nabla]\beta\|_{L^2} &\lesssim \|\nabla b\|_{L^\infty} \|\beta\|_{H^k} + \|\nabla \beta\|_{L^\infty} \|\Lambda^k b\|_{L^2} \\
    	&\lesssim \left( \|\nabla \bar{b}\|_{L^\infty} + \|\nabla \beta\|_{L^\infty}\right) \|\beta\|_{H^k} + \|\nabla \beta\|_{L^\infty} \|\bar{b}\|_{H^k}, \\
    	\|[\Lambda^k, b \cdot \nabla]w\|_{L^2} &\lesssim \|\nabla b\|_{L^\infty} \|w\|_{H^k} + \|\nabla w\|_{L^\infty} \|\Lambda^k b\|_{L^2} \\
    	&\lesssim \left(\|\nabla \bar{b}\|_{L^\infty} + \|\nabla \beta\|_{L^\infty}\right) \|w\|_{H^k} + \|\nabla w\|_{L^\infty} \left(\|\bar{b}\|_{H^k}+\|\beta\|_{H^k}\right).
    	\end{align*}
    	For the remaining terms, we employ Lemma \ref{lem:fractional_leibniz} to exploit the smallness of the $L^2$ norm of $(w,\beta)$. Consequently, we obtain
    	\begin{align*}
    	\|\Lambda^k(w \cdot \nabla \bar{u})\|_{L^2} &\lesssim \|w\|_{L^2} \|\nabla^{k+1} \bar{u}\|_{L^\infty} + \|\nabla \bar{u}\|_{L^\infty} \|w\|_{H^k}, \\
    	\|\Lambda^k(\beta \cdot \nabla \bar{u})\|_{L^2} &\lesssim \|\beta\|_{L^2} \|\nabla^{k+1} \bar{u}\|_{L^\infty} + \|\nabla \bar{u}\|_{L^\infty} \|\beta\|_{H^k}, \\
    	\|\Lambda^k(w \cdot \nabla \bar{b})\|_{L^2} &\lesssim \|w\|_{L^2} \|\nabla^{k+1} \bar{b}\|_{L^\infty} + \|\nabla \bar{b}\|_{L^\infty} \|w\|_{H^k}, \\
    	\|\Lambda^k(\beta \cdot \nabla \bar{b})\|_{L^2} &\lesssim \|\beta\|_{L^2} \|\nabla^{k+1} \bar{b}\|_{L^\infty} + \|\nabla \bar{b}\|_{L^\infty} \|\beta\|_{H^k}.
    	\end{align*}
    	Collecting the above bounds, we arrive at the energy inequality
    	\begin{align}
    	\frac{\mathrm{d}}{\mathrm{d}t} Y_k(t) & \lesssim
    	\left(\|\nabla \bar{u}\|_{L^\infty} + \|\nabla \bar{b}\|_{L^\infty} + \|\nabla w\|_{L^\infty} + \|\nabla \beta\|_{L^\infty} \right) Y_k(t) \notag \\
    	&\quad+ Y_0(t) \left( \|\nabla^{k+1} \bar{u}\|_{L^\infty} + \|\nabla^{k+1} \bar{b}\|_{L^\infty} \right)\notag  \\
    	&\quad+ \left( \|\nabla w\|_{L^\infty} + \|\nabla \beta\|_{L^\infty} \right) \left( \|\bar{u}\|_{H^k} + \|\bar{b}\|_{H^k} \right) + \left( \|\overline{E}_u\|_{H^k} + \|\overline{E}_b\|_{H^k} \right) \notag \\
    	&:= \left(\|\nabla \bar{u}\|_{L^\infty} + \|\nabla \bar{b}\|_{L^\infty} + \|\nabla w\|_{L^\infty} + \|\nabla \beta\|_{L^\infty} \right) Y_k(t)+S_k(t). \label{eq:Yk_ODE}
    	\end{align}
    	
    	Using \eqref{eq:approx_grad_bound} and the bootstrap assumption \eqref{bootstrap_assumption}, the coefficient of $Y_k(t)$ in \eqref{eq:Yk_ODE} is bounded by $3M_\varepsilon \mu^{2-s}\nu^{\frac{1}{2}}$. Moreover, by the $L^2$ estimate \eqref{eq:Y0_est}, the bootstrap assumption \eqref{bootstrap_assumption}, and Proposition \ref{prop:approximation}, we estimate the three terms in $S_k(t)$. The first term is bounded by $\left(\mu^{-s}(\mu^{-1}\nu)\right)\left(\mu^{k+2-s}\nu^{\frac{1}{2}}\right) = \mu^{k+1-2s}\nu^{\frac{3}{2}}$. The second term is bounded by $\left(\mu^{2-s}\nu^{\frac{1}{2}}\right)\left(\mu^{k-s}\right) = \mu^{k+2-2s}\nu^{\frac{1}{2}}$, which is the dominant scale. The third term is bounded by $\mu^{k-s}(\mu^{-1}\nu)\mu^{2-s}\nu^{\frac{1}{2}} = \mu^{k+1-2s}\nu^{\frac{3}{2}}$. Therefore, $S_k(t) \lesssim \mu^{k+2-2s}\nu^{\frac{1}{2}}$. Finally, applying Gronwall's inequality on $[0, t_0]$ gives
    	\begin{equation*}
    	Y_k(t) \lesssim \exp\left( 3M_\varepsilon \mu^{2-s}\nu^{\frac{1}{2}} t^* \right) \int^{t} S_k(\tau) \,\d\tau \lesssim \left(\varepsilon^{-N-2}\mu^{-2+s}\nu^{-\frac{1}{2}}\right)\left( \mu^{k+2-2s}\nu^{\frac{1}{2}} \right) = C_{k,\varepsilon} \mu^{k-s},
    	\end{equation*}
    	which implies \eqref{eq:Yk_est}.
    \end{proof}
    
    \noindent \textbf{Step 2: Bootstrap closure} \\
    	We are now in a position to close the bootstrap argument. Let $\delta \in (0, \frac{\gamma}{4})$ be fixed. By the Sobolev embedding $H^{\frac{5}{2}+\delta}(\R^3) \hookrightarrow W^{1,\infty}(\R^3)$ and an obvious interpolation, we deduce
    	\begin{equation*}
    	\|\nabla w(t)\|_{L^\infty} + \|\nabla \beta(t)\|_{L^\infty} \lesssim Y_{\frac{5}{2}+\delta}(t) \le Y_0(t)^{1-\alpha} Y_k(t)^\alpha,
    	\end{equation*}
    	where $\alpha = \frac{\frac{5}{2}+\delta}{k}$. Substituting the estimates \eqref{eq:Y0_est} and \eqref{eq:Yk_est}, and noting that $k\alpha = \frac{5}{2} + \delta$, the combined exponent of $\mu$ is given by
    	\begin{equation*}
    	(-s-\gamma)(1-\alpha) + (k-s)\alpha = \frac{5}{2} - s - \gamma + \delta + \gamma\alpha.
    	\end{equation*}
    	We compare this exponent to that of the bootstrap threshold $2M_\varepsilon \mu^{2-s}\nu^{\frac{1}{2}} = 2M_\varepsilon \mu^{\frac{5}{2}-s-\frac{\gamma}{2}}$. By choosing $k \gg \frac{5}{2}$ sufficiently large such that $\gamma\alpha < \frac{\gamma}{4}$, we  bound the exponent difference
    	\begin{equation*}
    	\left( \frac{5}{2} - s - \gamma + \delta + \gamma\alpha \right)- \left( \frac{5}{2} - s - \frac{\gamma}{2} \right) = -\frac{\gamma}{2} + \delta + \gamma\alpha < 0.
    	\end{equation*}
    	Consequently, for $\mu$ sufficiently large, we obtain
    	\begin{equation}\label{w-beta-t*}
    	\|\nabla w(t)\|_{L^\infty} + \|\nabla \beta(t)\|_{L^\infty} \ll M_\varepsilon \mu^{2-s}\nu^{\frac{1}{2}}.
    	\end{equation}
    	By a continuity argument, the bootstrap assumption \eqref{bootstrap_assumption} holds on the entire interval $[0,t^*]$.
    	
    	\noindent \textbf{Step 3: Final estimates}\\ 
    	To exclude finite-time blow-up for the exact solution $(u, b)$ before $t^*$, we combine \eqref{eq:approx_grad_bound} and \eqref{w-beta-t*} to obtain, for all $t\in[0,t^*]$,
    	\begin{align*}
    	\|\nabla u(t)\|_{L^\infty} + \|\nabla b(t)\|_{L^\infty} &\le \left( \|\nabla \bar{u}(t)\|_{L^\infty} + \|\nabla \bar{b}(t)\|_{L^\infty} \right) + \left( \|\nabla w(t)\|_{L^\infty} + \|\nabla \beta(t)\|_{L^\infty} \right) \\
    	&\le M_\varepsilon \mu^{2-s}\nu^{\frac{1}{2}} + 2M_\varepsilon \mu^{2-s}\nu^{\frac{1}{2}} = 3M_\varepsilon \mu^{2-s}\nu^{\frac{1}{2}}.
    	\end{align*}
    	Integrating over $[0, t^*]$ yields
    	\begin{equation*}
    	\int_0^{t^*} \left( \|\nabla u(\tau)\|_{L^\infty} + \|\nabla b(\tau)\|_{L^\infty} \right) \,\mathrm{d}\tau \le \int_0^{t^*} 3M_\varepsilon \mu^{2-s}\nu^{\frac{1}{2}} \,\mathrm{d}\tau \le 3M_\varepsilon \varepsilon^{-N-2}.
    	\end{equation*}
    	Since the quantity $3M_\varepsilon \varepsilon^{-N-2} < \infty$ is independent of $\mu$, the Beale-Kato-Majda type blow-up criterion ensures that the exact smooth solution $(u, b)$ does not blow up at $t^*$.
    	
    	Finally, interpolating between $Y_0$ and $Y_k$ yields the desired bound, for $t\in[0,t^*]$,
    	\begin{equation}\label{eq:Ys_final}
    	Y_s(t) \le Y_0(t)^{1-\frac{s}{k}} Y_k(t)^{\frac{s}{k}} \lesssim (\mu^{-s-\gamma})^{1-\frac{s}{k}} (\mu^{k-s})^{\frac{s}{k}} = \mu^{-\gamma + \frac{s\gamma}{k}}.
    	\end{equation}
    	By taking $k$ large enough such that $\frac{s\gamma}{k} \le \frac{\gamma}{2}$, we conclude that $\sup_{t\in[0,t^*]}Y_s(t) \lesssim \mu^{-\frac{\gamma}{2}}$.
    \end{proof}
    
        \section{Proof of Theorem \ref{th-ill-ideal}}\label{sec-proof}
        \begin{proof}[\normalfont\bfseries Proof of Theorem \ref{th-ill-ideal}]
        With Proposition \ref{prop:main_perturbation} established, we deduce the norm inflation of the exact magnetic field at $t = t^*$. According to Lemma \ref{lem:norm_inflation}, the approximate magnetic field satisfies
        \begin{equation}\label{bar-b-epsilon}
        \|\bar{b}(t^*)\|_{H^s(\R^3)} \ge \varepsilon^{-2}.
        \end{equation}
        Moreover, by Proposition \ref{prop:main_perturbation}, the difference in the magnetic field in $H^s$ satisfies
        \begin{equation}\label{beta-epsilon}
        \|\beta(t^*)\|_{H^{s}(\R^3)} \le C_{s,\varepsilon}\mu^{-\frac{\gamma}{2}},
        \end{equation}
        which implies that $\|\beta(t^*)\|_{H^s(\R^3)}\leq 1$ for sufficiently large $\mu$.
        
        By \eqref{bar-b-epsilon} and \eqref{beta-epsilon}, we obtain
        \begin{equation}
        \|b(t^*)\|_{H^s(\R^3)} \ge \|\bar{b}(t^*)\|_{H^s(\R^3)} - \|\beta(t^*)\|_{H^s(\R^3)} \ge \varepsilon^{-2} - 1.
        \end{equation}
        Since $\varepsilon > 0$ can be arbitrarily small, this proves the norm inflation for the ideal MHD equations \eqref{ideal mhd}.
        
        Finally, we claim that the velocity field $u_{\varepsilon}(t)$ remains uniformly bounded in $L^\infty([0, t^*]; H^s)$. Indeed, by definition, the exact velocity is $u = \bar{u} + w$. We estimate the two components separately.
        
        According to \eqref{eq:approx_Wkp}, for all $t\in[0,t^*]$, the approximate velocity satisfies
        \begin{align}\label{u-Hk-ideal}
        \|\bar{u}(t)\|_{H^k(\R^3)} \lesssim \varepsilon^2\mu^{k-s}.
        \end{align}
        Therefore, by \eqref{u-Hk-ideal}, we evaluate the space-time norms over $t \in [0, t^*]$
        \begin{align*}
        \|\bar{u}\|_{L^{\infty}([0,t^*]; H^{s})}&= \sup_{t\in[0,t^*]}\|\bar{u}(t)\|_{H^{s}} \lesssim \varepsilon^2.
        \end{align*}
        For the difference in the velocity field $w$, Proposition \ref{prop:main_perturbation} yields that
        \begin{align*}
        \|w\|_{L^{\infty}([0,t^*]; H^{s})}&= \sup_{t\in[0,t^*]}\|w(t)\|_{H^{s}} \lesssim \mu^{-\frac{\gamma}{2}}.
        \end{align*}
        We now combine the above estimates to obtain
        \begin{align*}
        \|u\|_{L^\infty([0,t^*]; H^s)} \le \|\bar{u}\|_{L^\infty([0,t^*]; H^s)} + \|w\|_{L^\infty([0,t^*]; H^s)} \le C\left(\varepsilon^2 +\mu^{-\frac{\gamma}{2}}\right)\le \varepsilon,
        \end{align*}
        where the last inequality follows by choosing $\mu$ sufficiently large.
        
        Combining the above estimates with the initial bounds completes the proof of Theorem \ref{th-ill-ideal}.
       \end{proof}

    \section{Proof of Theorem \ref{th-ill-mhd-unified}}\label{sec-proof2}
    In this section, we sketch the proof of Theorem \ref{th-ill-mhd-unified} for the three dissipative variants of the 3D MHD equations. The constructions and perturbation arguments largely follow those of the ideal case. Over the short time interval $[0,t^*]$, the dissipative terms can be treated as perturbative errors. The main modifications are the adjustment of the initial amplitudes to match the Sobolev regularities $(s_u,s_b)$ and, when necessary, the use of weighted energy estimates to handle the regularity mismatch.

    \subsection{The non-resistive case: $H^s \times H^{s+1}$ for $0 < s < \frac{1}{2}$}\label{subsec-non-resistive case}
    
    For the non-resistive MHD equations, the magnetic field $b$ has one more derivative than the velocity field $u$. For $0<s<\frac12$, we choose the parameters
    \begin{equation}\label{s-non-resistive}
    \begin{cases}
    \gamma=\frac{\frac{1}{2}-s}{100}>0,\\
    \nu = \mu^{1-\gamma}.
    \end{cases}
    \end{equation}
    This choice guarantees the following inequality in the viscous regime:
    \begin{equation}\label{crucial_relation}
    \mu^2 \le \mu^{-1}\nu (\mu^{2-s}\nu^{\frac{1}{2}}).
    \end{equation}
    Indeed, the relation \eqref{crucial_relation} is equivalent to $\mu^{1+s} \le \nu^{\frac{3}{2}}$, which holds since $1+s < \frac{3}{2}(1-\gamma)$ by \eqref{s-non-resistive}.
    
    To reflect this regularity mismatch and ensure
    \[
    \|u_0\|_{H^s(\R^3)}+\|b_0\|_{H^{s+1}(\R^3)}\le \varepsilon,
    \]
    we modify the leading-order profiles as follows:
    \begin{equation*}
    \begin{cases}
    u_{0,\theta}(z,r)=0\\
    u_{0,r}(z,r)= -\varepsilon^2 \mu^{1-s}\nu^{\frac12}f_u'(\mu \rho)\partial_z\rho\\
    u_{0,z}(z,r)= \varepsilon^2 \mu^{1-s}\nu^{\frac12}f_u'(\mu \rho)\partial_r\rho
    \end{cases}
    \end{equation*}
    and
    \begin{equation*}
    \begin{cases}
    b_{0,\theta}(z,r)= \varepsilon^2 \mu^{-s}\nu^{\frac12}g_b(\mu \rho)\sin (\phi)\\
    b_{0,r}(z,r)=0\\
    b_{0,z}(z,r)=0
    \end{cases}
    \end{equation*}
    where $f_u$ and $g_b$ are as chosen in \eqref{suppfg}. Here, compared with the ideal case, the magnetic amplitude is smaller by a factor of $\mu^{-1}$.
    
    The critical time is defined as before by
    \[
    t^*=\varepsilon^{-N-2}\mu^{-2+s}\nu^{-\frac12}.
    \]
    
    In the non-resistive case, the approximate solutions $(\bar{u}, \bar{b})$ satisfy the following system:
    \begin{equation}
    \begin{cases}
    \partial_t \bar{u} - \Delta \bar{u} + \bar{u} \cdot \nabla \bar{u} - \bar{b} \cdot \nabla \bar{b} + \nabla \bar{P} = \overline{E}_{u} + \overline{E}_{vis}, \\
    \partial_t \bar{b} + \bar{u} \cdot \nabla \bar{b} - \bar{b} \cdot \nabla \bar{u} = \overline{E}_{b},\\
    (\bar{u},\bar{b})|_{t = 0} = (u_{0}, b_{0}),
    \end{cases}
    \end{equation}
    where $\bar{P}=p_0(r,z)$ is the approximate pressure from \eqref{pressure(r,z)}, $\overline{E}_u$ and $\overline{E}_b$ are the same as those in Section \ref{sec-approx}, and the new viscous error is given by $\overline{E}_{vis}:= -\Delta \bar{u}$. 
    
    Similar to the construction method and estimation of the approximate solution in the ideal case, we have estimates for $(\bar{u}, \bar{b})$ on the interval $t \in [0, t^*]$
    \begin{equation}\label{eq:approx_Wkp2}
    \begin{aligned}
    \|\bar{u}(t)\|_{W^{k,p}} &\le C_{k,p} \varepsilon^2 \mu^{k-s} \mu^{1-\frac{2}{p}} \nu^{\frac{1}{2}-\frac{1}{p}}, \\
    \|\bar{b}(t)\|_{W^{k,p}} &\le C_{k,p} \varepsilon^{2-kN} \mu^{k-(s+1)} \mu^{1-\frac{2}{p}} \nu^{\frac{1}{2}-\frac{1}{p}}.
    \end{aligned}
    \end{equation}
    
    Now, we need to justify that treating $\overline{E}_{vis}$ as a negligible error is compatible with the previous estimates in \eqref{error_L2}. Since $$\|\nabla^k \overline{E}_{vis}(t)\|_{L^2(\R^3)} = \|\nabla^{k+2} \bar{u}\|_{L^2(\R^3)} \lesssim \mu^{k+2-s},$$ it follows from \eqref{crucial_relation} that the total error is bounded by
    \begin{align*}
    \|\nabla^k\overline{E}_{u}(t)\|_{L^2(\R^3)}+\|\nabla^k\overline{E}_{vis}(t)\|_{L^2(\R^3)} &\lesssim \mu^{k-s}\left( \mu^{-1}\nu(\mu^{2-s}\nu^{\frac{1}{2}})\right)\\
    \|\nabla^k\overline{E}_{b}(t)\|_{L^2(\R^3)} &\lesssim \mu^{k-(s+1)}\left( \mu^{-1}\nu(\mu^{2-s}\nu^{\frac{1}{2}})\right).
    \end{align*}
    Thus, the viscous dissipation acts as an error.
    
    \begin{lemma}\label{lem:norm_inflation2}
    	There exists a universal constant $\varepsilon_0 \in (0, 1)$ such that for any $0 < \varepsilon \le \varepsilon_0$, there holds the lower bound
    	\begin{equation*}
    	\|\bar{b}(t^*)\|_{H^{s+1}(\R^3)} \ge \varepsilon^{-2}
    	\end{equation*}
    	where $t^*>0$ is the critical time.
    \end{lemma}
    In the case of the non-resistive MHD equations, we obtain  the following system for $(w,\beta)$
    \begin{equation}\label{eq:perturbation_system2}
    \begin{cases}
    \partial_t w -\Delta w+ u \cdot \nabla w +w \cdot \nabla \bar{u} + \nabla q = b \cdot \nabla \beta+ \beta \cdot \nabla \bar{b} - \overline{E}_u-\overline{E}_{vis}, \\
    \partial_t \beta + u \cdot \nabla \beta + w \cdot \nabla \bar{b}= b \cdot \nabla w + \beta \cdot \nabla \bar{u} - \overline{E}_b, \\
    \operatorname{div}w=0, \quad \operatorname{div}\beta=0, \\
    (w, \beta)|_{t=0} = (0, 0).
    \end{cases}
    \end{equation}
    
    Similar to \eqref{bootstrap_assumption} in Section \ref{sec-pertur}, we adjust the bootstrap assumption as follows.
    
    \noindent \textbf{Bootstrap assumption:} 
    Let $t_0\in(0,t^*]$ be the maximal time such that
    \begin{equation}\label{bootstrap_assumption2}
    \|\nabla w\|_{L^\infty} + \mu\|\nabla\beta\|_{L^\infty} \le 2M_\varepsilon \mu^{2-s}\nu^{\frac{1}{2}}, \quad \text{for all } t \in [0, t_0].
    \end{equation}
    Now, we define the following weighted energy functional
    \begin{equation*}
    Y_m(t)=\left(\|w(t)\|_{H^m(\R^3)}^2+\mu^2\|\beta(t)\|_{H^m(\R^3)}^2\right)^{\frac12}.
    \end{equation*}
    Equipped with this weighted energy functional, we derive the foundational  Sobolev estimates.
    
    \begin{lemma}\label{lem:Hk_estimate2}
    	For any real number $k > \frac{5}{2}$, under the bootstrap assumption there exist constants $C_\varepsilon, C_{k,\varepsilon} > 0$ independent of $\mu$ and $\nu$ such that for all $t \in [0, t_0]$, the following estimates hold:
    	\begin{align}
    	Y_0(t) &\le C_\varepsilon \mu^{-s}(\mu^{-1}\nu), \label{eq:Y0_est_nonresistive} \\
    	Y_k(t) &\le C_{k,\varepsilon} \mu^{k-s}. \label{eq:Yk_est_nonresistive}
    	\end{align}
    \end{lemma}

    With the energy estimates established, we first close the bootstrap argument. We then apply an interpolation argument to obtain the following main proposition for the non-resistive case.
    \begin{prop}\label{prop:main_perturbation2}
    	Let $T = T(u_0, b_0) > 0$ be the maximal time of existence for the local-in-time smooth solution $(u, b)$ of the non-resistive MHD equations with the initial data defined by \eqref{ini-data}. For any $\varepsilon > 0$, there exists $\mu_0 > 0$ sufficiently large such that if $\mu \ge \mu_0$, then $0 < t^* < T$, namely, $(u, b) \in C([0,t^*];H^\infty(\R^3)\times H^\infty(\R^3))$.
    	
    	More quantitatively, for any $\varepsilon > 0$, if $\mu \ge \mu_0$, then
    	\begin{equation}\label{w+beta-ineq2}
    	\|w\|_{L^\infty([0,t^*]; H^s)}+\|\beta\|_{L^\infty([0,t^*]; H^{s+1})}\le C_{s,\varepsilon}\mu^{-\frac{\gamma}{2}},
    	\end{equation}
    	where $C_{s,\varepsilon}$ is independent of $\mu$.
    \end{prop}

      By Lemma \ref{lem:norm_inflation2} and Proposition \ref{prop:main_perturbation2}, the exact magnetic field satisfies, at the critical time $t=t^\ast$, 
      \[ \|b(t^\ast)\|_{H^{s+1}} \geq \|\bar b(t^\ast)\|_{H^{s+1}} -\|\beta(t^\ast)\|_{H^{s+1}} \geq \varepsilon^{-2}-1 \geq \varepsilon^{-1}, \] 
      provided $\mu$ is sufficiently large. This proves the magnetic norm inflation.

    	 Now, we verify that the velocity field $u(t)$ remains small not only in $L^\infty([0, t^*]; H^s)$, but also in $L^2([0, t^*]; H^{s+1})\cap L^1([0, t^*]; H^{s+2})$. Indeed, by the decomposition $u = \bar{u} + w$, we estimate the two components separately.

According to \eqref{eq:approx_Wkp2}, for all $t\in[0,t^*]$, the approximate velocity satisfies
\begin{align}\label{u-Hk}
	\|\bar{u}(t)\|_{H^k(\R^3)} \lesssim \varepsilon^2\mu^{k-s}.
\end{align}
By the definition of $t^*$ and \eqref{s-non-resistive}, we have
\begin{align*}
t^*\mu^2=\varepsilon^{-N-2}\mu^{s-\frac12+\frac{\gamma}{2}}, \qquad s-\frac12+\frac{\gamma}{2}<0.
\end{align*}
Fix $0<\varepsilon\le\varepsilon_0$. We then choose $\mu$ sufficiently large depending on $\varepsilon$ such that
\begin{align}\label{mu-choice-non-resistive}
\mu \gg \nu \gg \varepsilon^{-1} \gg 1, \qquad t^*\mu^2\leq1,
\qquad \mu^{-\frac{\gamma}{2}}\lesssim\varepsilon^2.
\end{align}
Therefore, by \eqref{u-Hk} and \eqref{mu-choice-non-resistive}, we evaluate the space-time norms over $t \in [0, t^*]$
\begin{align*}
\|\bar{u}\|_{L^{\infty}([0,t^*]; H^{s})}&= \sup_{t\in[0,t^*]}\|\bar{u}(t)\|_{H^{s}} \lesssim \varepsilon^2,\\
\|\bar{u}\|_{L^2([0,t^*]; H^{s+1})}&= (t^*)^{\frac{1}{2}} \|\bar{u}(t)\|_{H^{s+1}} \lesssim \varepsilon^2(t^*\mu^2)^{\frac{1}{2}}\lesssim \varepsilon^2,\\
\|\bar{u}\|_{L^1([0,t^*]; H^{s+2})}&= t^* \|\bar{u}(t)\|_{H^{s+2}} \lesssim \varepsilon^2(t^*\mu^2)\lesssim \varepsilon^2.
\end{align*}

For the difference in the velocity field $w$, Proposition \ref{prop:main_perturbation2} yields that, for all $t\in[0,t^*]$,
\begin{align}\label{w-Hk}
	\|w(t)\|_{H^s} \lesssim \mu^{-\frac{\gamma}{2}}.
\end{align}
Therefore, by \eqref{w-Hk}, we have the space-time norms over $t \in [0, t^*]$
\begin{align*}
\|w\|_{L^{\infty}([0,t^*]; H^{s})}&= \sup_{t\in[0,t^*]}\|w(t)\|_{H^{s}} \lesssim \mu^{-\frac{\gamma}{2}}\lesssim\varepsilon^2.
\end{align*} 
    Furthermore, from \eqref{eq:Y0_est_nonresistive} and \eqref{eq:Yk_est_nonresistive} in Lemma \ref{lem:Hk_estimate2}, we evaluate
    \begin{align*}
    \|w(t)\|_{H^{s+1}} &\le \|w(t)\|_{L^2}^{1-\frac{s+1}{k}} \|w(t)\|_{H^k}^{\frac{s+1}{k}} \lesssim (\mu^{-s-\gamma})^{1-\frac{s+1}{k}} (\mu^{k-s})^{\frac{s+1}{k}} = \mu^{1-\gamma + \frac{(s+1)\gamma}{k}}, \\
    \|w(t)\|_{H^{s+2}} &\le \|w(t)\|_{L^2}^{1-\frac{s+2}{k}} \|w(t)\|_{H^k}^{\frac{s+2}{k}} \lesssim (\mu^{-s-\gamma})^{1-\frac{s+2}{k}} (\mu^{k-s})^{\frac{s+2}{k}} = \mu^{2-\gamma + \frac{(s+2)\gamma}{k}},
    \end{align*}
    from which it follows that
    \begin{align*}
    \|w\|_{L^2([0,t^*]; H^{s+1})} &\lesssim (t^*)^{\frac{1}{2}} \mu^{1-\gamma + \frac{(s+1)\gamma}{k}} = (t^*\mu^2)^{\frac{1}{2}} \mu^{-\gamma + \frac{(s+1)\gamma}{k}}\lesssim (t^*\mu^2)^{\frac12}\mu^{-\frac{\gamma}{2}}\lesssim\varepsilon^2,\\
    \|w\|_{L^1([0,t^*]; H^{s+2})} & \lesssim t^* \mu^{2-\gamma + \frac{(s+2)\gamma}{k}} = (t^*\mu^2) \mu^{-\gamma + \frac{(s+2)\gamma}{k}}\lesssim (t^*\mu^2)\mu^{-\frac{\gamma}{2}}\lesssim\varepsilon^2,
    \end{align*}
    where we have chosen $k$ sufficiently large so that $\frac{(s+1)\gamma}{k}$ and $\frac{(s+2)\gamma}{k}$ are less than $\frac{\gamma}{2}$, used \eqref{mu-choice-non-resistive}. Combining the above estimates and $u=\bar{u}+w$, we obtain
    \begin{align*}
    &\|u(t)\|_{L^\infty_{t^*}H^s\cap L^2_{t^*}H^{s+1}\cap L^1_{t^*}H^{s+2}}\le \|\bar{u}(t)\|_{L^\infty_{t^*}H^s\cap L^2_{t^*}H^{s+1}\cap L^1_{t^*}H^{s+2}} + \|w(t)\|_{L^\infty_{t^*}H^s\cap L^2_{t^*}H^{s+1}\cap L^1_{t^*}H^{s+2}} \le \varepsilon,
    \end{align*}
    
    In conclusion, we have
    \begin{equation*}
    \begin{dcases}
    \|b(t^\ast)\|_{H^{s+1}}\geq \frac{1}{\varepsilon},\\
    \|u\|_{L^\infty([0, t^*]; H^s)\cap L^2([0, t^*]; H^{s+1})\cap L^1([0, t^*]; H^{s+2})}\le \varepsilon,
    \end{dcases}
    \qquad 0<t^*\leq\varepsilon.
    \end{equation*}    
    This completes the proof in the non-resistive case.

    \begin{remark}\label{rem:space_time_meaning}
    	We point out that in the supercritical Sobolev regime $0 < s <\frac{1}{2}$, the exact velocity field remains small in the space-time norms $L^2([0,t^*]; H^{s+1}(\R^3))$ and $L^1([0,t^*]; H^{s+2}(\R^3))$ required by Fefferman et al. \cite{fefferman2017local}. Nevertheless, the magnetic field still undergoes norm inflation in $H^{s+1}$.
    \end{remark}

    \subsection{The non-viscous case: $H^s \times H^{s-1}$ for $1<s < \frac{5}{2}$}
    For the non-viscous MHD equations, the velocity field $u$ has one more derivative than the magnetic field $b$. For $1 < s < \frac{5}{2}$, we choose the parameters
    \begin{equation}\label{s-non-viscous}
    \begin{cases}
    \gamma = \frac{\frac{5}{2}-s}{100} > 0, \\
    \nu = \mu^{1-\gamma}.
    \end{cases}
    \end{equation}
    To ensure that
    \[
    \|u_0\|_{H^s(\R^3)} + \|b_0\|_{H^{s-1}(\R^3)} \le \varepsilon,
    \]
    we modify the leading-order profiles as follows:
    \begin{equation*}
    \begin{cases}
    u_{0,\theta}(z,r) = \varepsilon^2 \mu^{1-s}\nu^{\frac12}f_u(\mu \rho)\sin (\phi) \\
    u_{0,r}(z,r) = -\varepsilon^2 \mu^{1-s}\nu^{\frac12}f_u'(\mu \rho)\partial_z\rho\\
    u_{0,z}(z,r) = \varepsilon^2 \mu^{1-s}\nu^{\frac12}f_u'(\mu \rho)\partial_r\rho
    \end{cases}
    \end{equation*}
    and
    \begin{equation*}
    \begin{cases}
    b_{0,\theta}(z,r) = \varepsilon^2 \mu^{-M}\nu^{\frac12}g_b(\mu \rho)\sin (\phi) \\
    b_{0,r}(z,r) = 0 \\
    b_{0,z}(z,r) = 0
    \end{cases}
    \end{equation*}
    where $M\ge 10$ is a fixed constant, and $f_u$ and $g_b$ are as chosen in \eqref{suppfg}.
    
    The critical time is defined identically as before by
    \begin{align*}
     t^* = \varepsilon^{-N-2}\mu^{-2+s}\nu^{-\frac12}.
    \end{align*}
  
    In the non-viscous case, we define the approximate solutions $(\bar{u}, \bar{b})$ as follows
    \begin{equation*}
    \begin{aligned}
    \bar{u}(t, x) &=  \bar{u}_\theta\mathbf{e}_\theta+u_{0,r} \mathbf{e}_r + (u_{0,z} + u_c) \mathbf{e}_z, \\
    \bar{b}(t, x) &= \bar{b}_\theta \mathbf{e}_\theta,
    \end{aligned}
    \end{equation*}
    where the toroidal velocity $\bar{u}_\theta$ satisfies the linear transport equation
    \begin{equation*}
    \begin{cases}
    \partial_t \bar{u}_\theta + (u_{0,r} \partial_r + u_{0,z} \partial_z) \bar{u}_\theta = 0,\\ 
    \bar{u}_\theta|_{t=0} = u_{0,\theta},
    \end{cases}
    \end{equation*}
    and the toroidal magnetic field $\bar{b}_\theta$ satisfies
    \begin{equation*}
    \begin{cases}
    \partial_t \bar{b}_\theta + (u_{0,r} \partial_r + u_{0,z} \partial_z) \bar{b}_\theta = 0,\\ 
    \bar{b}_\theta|_{t=0} = b_{0,\theta}.
    \end{cases}
    \end{equation*}
    
    Moreover, the approximate solutions $(\bar{u}, \bar{b})$ satisfy the following system:
    \begin{equation*}
    \begin{cases}
    \partial_t \bar{u} + \bar{u} \cdot \nabla \bar{u} - \bar{b} \cdot \nabla \bar{b} + \nabla \bar{P} = \overline{E}_{u}, \\
    \partial_t \bar{b} - \Delta \bar{b} + \bar{u} \cdot \nabla \bar{b} - \bar{b} \cdot \nabla \bar{u} = \overline{E}_{b} + \overline{E}_{res}, \\
    (\bar{u},\bar{b})|_{t = 0} = (u_{0}, b_{0}),
    \end{cases}
    \end{equation*}
    where $\bar{P}=p_0(r,z)$ is the approximate pressure from \eqref{pressure(r,z)}. After canceling the leading-order transport terms and the steady 2D Euler interactions, the non-zero components of the residual error fields $\overline{E}_u$ and $\overline{E}_b$ are given by
    \begin{equation}\label{non_viscous_error_components}
    \begin{cases}
    \overline{E}_{u,\theta} = u_c \partial_z \bar{u}_\theta + \frac{u_{0,r} \bar{u}_\theta}{r}, \\[8pt]
    \overline{E}_{u,r} = u_c \partial_z u_{0,r} - \frac{\bar{u}_\theta^2}{r} + \frac{\bar{b}_\theta^2}{r}, \\[8pt]
    \overline{E}_{u,z} = u_{0,r} \partial_r u_c + u_{0,z} \partial_z u_c + u_c \partial_z u_{0,z} + u_c \partial_z u_c, \\[8pt]
    \overline{E}_{b,\theta} = u_c \partial_z \bar{b}_\theta - \frac{u_{0,r} \bar{b}_\theta}{r},
    \end{cases}
    \end{equation}
    and the new resistive error is given by $\overline{E}_{res} := -\Delta \bar{b}$.
    
    Similar to the estimation of the approximate solutions in the ideal case, we have estimates for $(\bar{u}, \bar{b})$ on the interval $t \in[0, t^*]$
    \begin{equation}\label{eq:approx_Wkp3}
    \begin{aligned}
    \|\bar{u}(t)\|_{W^{k,p}(\R^3)} &\le C_{k,p} \varepsilon^{2-kN}\mu^{k-s} \mu^{1-\frac{2}{p}} \nu^{\frac{1}{2}-\frac{1}{p}}, \\
    \|\bar{b}(t)\|_{W^{k,p}(\R^3)} &\le C_{k,p} \varepsilon^{2-kN}\mu^{k-M-1} \mu^{1-\frac{2}{p}} \nu^{\frac{1}{2}-\frac{1}{p}}.
    \end{aligned}
    \end{equation}
    
    Now, we need to justify that treating $\overline{E}_{res}$ as a negligible error is compatible with the previous estimates. Since 
    $$\|\nabla^k \overline{E}_{res}(t)\|_{L^2(\R^3)} = \|\nabla^{k+2} \bar{b}\|_{L^2(\R^3)} \lesssim \mu^{k-M+1},$$ 
    it follows  that the total error \eqref{non_viscous_error_components} is bounded by
    \begin{align*}
    \|\nabla^k\overline{E}_{u}(t)\|_{L^2(\R^3)} &\lesssim \mu^{k-s}\left( \mu^{-1}\nu(\mu^{2-s}\nu^{\frac{1}{2}})\right), \\
    \|\nabla^k\overline{E}_{b}(t)\|_{L^2(\R^3)} + \|\nabla^k\overline{E}_{res}(t)\|_{L^2(\R^3)} &\lesssim \mu^{k-M-s}\nu^{\frac{3}{2}}+\mu^{k-M+1}\\
    &\lesssim \mu^{k-M+1}.
    \end{align*}
    Thus, the resistive dissipation acts as an error.
    
    \begin{lemma}\label{lem:norm_inflation3}
    	There exists a universal constant $\varepsilon_0 \in (0, 1)$ such that for any $0 < \varepsilon \le \varepsilon_0$, there holds the lower bound
    	\begin{equation*}
    	\|\bar{u}(t^*)\|_{H^{s}(\R^3)} \ge \varepsilon^{-2}
    	\end{equation*}
    	where $t^*>0$ is the critical time.
    \end{lemma}
    
    In the case of the non-viscous MHD equations, we obtain the following system for $(w,\beta)$
    \begin{equation}\label{eq:perturbation_system3}
    \begin{cases}
    \partial_t w + u \cdot \nabla w + w \cdot \nabla \bar{u} + \nabla q = b \cdot \nabla \beta + \beta \cdot \nabla \bar{b} - \overline{E}_u, \\
    \partial_t \beta - \Delta \beta + u \cdot \nabla \beta + w \cdot \nabla \bar{b} = b \cdot \nabla w + \beta \cdot \nabla \bar{u} - \overline{E}_b-\overline{E}_{res}, \\
    \operatorname{div}w=0, \quad \operatorname{div}\beta=0, \\
    (w, \beta)|_{t=0} = (0, 0).
    \end{cases}
    \end{equation}
    
    Similar to the ideal case, we also define the bootstrap assumption as follows.
    
    \noindent \textbf{Bootstrap assumption:} 
    Let $t_0\in(0,t^*]$ be the maximal time such that
    \begin{equation}\label{bootstrap_assumption3}
    \|\nabla w(t)\|_{L^\infty(\R^3)} + \|\nabla \beta(t)\|_{L^\infty(\R^3)} \le 2M_\varepsilon \mu^{2-s}\nu^{\frac{1}{2}}, \quad \text{for all } t \in [0, t_0].
    \end{equation}
    
    Now, we define the following energy functional:
    \begin{equation*}
   Y_m(t)=\left(\|w(t)\|_{H^m(\R^3)}^2+\|\beta(t)\|_{H^m(\R^3)}^2\right)^{\frac12}.
    \end{equation*}
    
    Equipped with this energy functional, we derive the following Sobolev estimates.

    \begin{lemma}\label{lem:Hk_non_viscous}
    	For any real number $k > \frac{5}{2}$, under the bootstrap assumption \eqref{bootstrap_assumption3}, there exist constants $C_\varepsilon, C_{k,\varepsilon} > 0$ independent of $\mu$ and $\nu$ such that for all $t \in [0, t_0]$,
    	\begin{align}
    	Y_0(t) &\le C_\varepsilon \mu^{-s}(\mu^{-1}\nu), \label{eq:Y0_est_nonviscous} \\
    	Y_k(t) &\le C_{k,\varepsilon} \mu^{k-s}. \label{Hk_estimate3}
    	\end{align}
    \end{lemma}

    With the energy estimates established, we state the main proposition for the non-viscous case, closing the bootstrap argument.
    
    \begin{prop}\label{prop:main_perturbation3}
    	Let $T = T(u_0, b_0) > 0$ be the maximal time of existence for the local-in-time smooth solution $(u, b)$ of the non-viscous MHD equations with the given initial data. For any $\varepsilon > 0$, there exists $\mu_0 > 0$ sufficiently large such that if $\mu \ge \mu_0$, then $0 < t^* < T$, namely, $(u, b) \in C([0,t^*];H^\infty(\R^3)\times H^\infty(\R^3))$.
    	
    	More quantitatively, for any $\varepsilon > 0$, if $\mu \ge \mu_0$, then
    	\begin{equation}\label{w+beta-ineq3}
    	\|w\|_{L^\infty([0,t^*]; H^s)}+\|\beta\|_{L^\infty([0,t^*]; H^s)}\le C_{s,\varepsilon}\mu^{-\frac{\gamma}{2}},
    	\end{equation}
    	where $C_{s,\varepsilon}$ is independent of $\mu$.
    \end{prop}

    By the definitions of $t^*$ and \eqref{s-non-viscous}, we have
    \begin{align*}
    t^*
    =\varepsilon^{-N-2}\mu^{s-\frac52+\frac{\gamma}{2}},
    \qquad
    s-\frac52+\frac{\gamma}{2}<0.
    \end{align*}
    Fix $0<\varepsilon\leq\varepsilon_0$. We then choose $\mu$ sufficiently
    large depending on $\varepsilon$ such that
    \begin{align}\label{parameter-non-viscous}
    \mu\gg\nu\gg\varepsilon^{-1}\gg1,
    \qquad
    0<t^*\leq\varepsilon,
    \qquad
    \mu^{-\frac{\gamma}{2}}\lesssim\varepsilon^2.
    \end{align}
    
    We now derive the conclusion for the non-viscous case. By Lemma \ref{lem:norm_inflation3} and Proposition \ref{prop:main_perturbation3}, 
    \[ \|u(t^\ast)\|_{H^s} \geq \|\bar u(t^\ast)\|_{H^s}-\|w(t^\ast)\|_{H^s} \geq \varepsilon^{-2}-1 \geq \varepsilon^{-1}, \] 
    for sufficiently large $\mu$. Hence the norm inflation occurs in the velocity field.

    Next, we evaluate the magnetic field in the space-time norms $L^\infty([0, t^*]; H^{s-1})\cap L^1([0, t^*]; H^{s+1})$. From \eqref{eq:approx_Wkp3}, for any $k \ge 0$, we have
    \begin{equation}\label{eq:b_bar_bound}
    \|\bar{b}\|_{L^\infty([0, t^*]; H^k(\R^3))} \le C_{k,\varepsilon} \mu^{k-M-1}.
    \end{equation}
    Moreover, Proposition \ref{prop:main_perturbation3} ensures that the difference in the magnetic field $\beta$ obeys 
    \begin{equation*}
    \|\beta\|_{L^\infty([0,t^*]; H^{s-1}(\R^3))}\le\|\beta\|_{L^\infty([0, t^*]; H^s(\R^3))} \le C_{s,\varepsilon}\mu^{-\frac{\gamma}{2}}.
    \end{equation*}

   Furthermore, we evaluate the space-time norms of $\beta$. By integrating the magnetic dissipation term $-\Delta \beta$ in the energy estimate underlying Lemma \ref{lem:Hk_non_viscous} over $[0, t^*]$, one extracts the space-time bounds
   \begin{equation}\label{eq:beta_L2H1_L2Hk1}
   \|\beta\|_{L^2([0,t^*]; H^1)} \lesssim \mu^{-s-\gamma}, \quad \text{and} \quad \|\beta\|_{L^2([0,t^*]; H^{k+1})} \lesssim \mu^{k-s}.
   \end{equation} 
   By \eqref{eq:beta_L2H1_L2Hk1} and interpolating, we deduce
   \begin{equation*}
   \|\beta\|_{L^2([0,t^*]; H^{s+1})} \le \|\beta\|_{L^2([0,t^*]; H^1)}^{1-\frac{s}{k}} \|\beta\|_{L^2([0,t^*]; H^{k+1})}^{\frac{s}{k}} \lesssim \left(\mu^{-s-\gamma}\right)^{1-\frac{s}{k}} \left(\mu^{k-s}\right)^{\frac{s}{k}} = \mu^{-\gamma + \frac{s\gamma}{k}},
   \end{equation*}
   from which it follows by applying the Cauchy-Schwarz inequality in time that
   \begin{align*}
   \|\beta\|_{L^1([0,t^*]; H^{s+1})} &\le (t^*)^{\frac{1}{2}} \|\beta\|_{L^2([0,t^*]; H^{s+1})} \\
   &\lesssim  \mu^{\frac{2s-5}{4} - \frac{3\gamma}{4} + \frac{s\gamma}{k}} \lesssim \mu^{\frac{2s-5}{4}} \lesssim\mu^{-\frac{\gamma}{2}}\lesssim\varepsilon^2,
   \end{align*}   
   where $k$ is chosen sufficiently large such that
   $\frac{s\gamma}{k}\leq\frac{3\gamma}{4}$, and we have used
   $\frac{2s-5}{4}=-50\gamma<-\frac{\gamma}{2}$.

   Since $M\geq10$, it follows from \eqref{parameter-non-viscous} and
   \eqref{eq:b_bar_bound} that
   \begin{align*}
   &\|\bar b\|_{L^\infty([0,t^*];H^{s-1})}
   \lesssim_{\varepsilon}\mu^{s-M-2}
   \lesssim_{\varepsilon}\mu^{-\frac{\gamma}{2}}
   \leq\varepsilon^2,\\
   &\|\bar b\|_{L^1([0,t^*];H^{s+1})}
   \leq t^*\|\bar b\|_{L^\infty([0,t^*];H^{s+1})}\leq\varepsilon^2.
   \end{align*}

   Finally, combining the above estimates and $b=\bar{b}+\beta$, we obtain
   \begin{equation*}
   \|b\|_{L^\infty([0,t^*]; H^{s-1})} \le \|\bar{b}\|_{L^\infty([0,t^*]; H^{s-1})} + \|\beta\|_{L^\infty([0,t^*]; H^{s-1})}\lesssim\varepsilon^2
   \end{equation*}
   and
   \begin{equation*}
   \|b\|_{L^1([0,t^*]; H^{s+1})} \le \|\bar{b}\|_{L^1([0,t^*]; H^{s+1})} + \|\beta\|_{L^1([0,t^*]; H^{s+1})}\lesssim\varepsilon^2.
   \end{equation*}
   In conclusion, we have
   \begin{equation*}
   \begin{dcases}
   \|u(t^*)\|_{H^s(\mathbb R^3)}
   \geq \frac{1}{\varepsilon},\\
   \|b\|_{L^\infty([0,t^*];H^{s-1}(\R^3))
   	\cap L^1([0,t^*];H^{s+1}(\R^3))}
   \leq \varepsilon,
   \end{dcases}
   \qquad 0<t^*\leq\varepsilon.
   \end{equation*}
   This completes the proof in the non-viscous case.

    \begin{remark}
    	We explain why the Magnetic-solo ansatz used in the ideal and non-resistive cases fails here. Indeed, if the velocity and magnetic fields were chosen with the same amplitudes as in the ideal construction, then, on the time interval $[0,t^*]$ with $0<t^*\leq\varepsilon$,
    	\begin{align*}
    	\|\bar{u} \cdot \nabla \bar{b}\|_{L^1([0,t^*];H^{s-1})} + \|\bar{b} \cdot \nabla \bar{u}\|_{L^1([0,t^*];H^{s-1})} \le C_{\varepsilon}\mu^{-1}.
    	\end{align*}
    	By choosing $\mu$ sufficiently large, the two convection terms in the magnetic equation remain small and cannot produce magnetic norm inflation. We therefore suppress the magnetic field and instead obtain norm inflation in the velocity field.
    \end{remark}

    \subsection{The viscous and resistive case: $H^s \times H^s$ for $0 < s < \frac{1}{2}$}
    For the viscous and resistive MHD equations, the regularity spaces are matched. Thus we return to the same initial profiles used for the ideal case:
    \begin{equation*}
    \begin{cases}
    u_{0,\theta}(z,r) =0 \\
    u_{0,r}(z,r) = -\varepsilon^2 \mu^{1 - s}\nu^{\frac{1}{2}}f_u'(\mu \rho)\partial_z\rho \\
    u_{0,z}(z,r) = \varepsilon^2 \mu^{1 - s}\nu^{\frac{1}{2}}f_u'(\mu \rho)\partial_r\rho \\
    \end{cases}
    \end{equation*}
    and
    \begin{equation*}
    \begin{cases}
    b_{0,\theta}(z,r) = \varepsilon^2 \mu^{1 - s}\nu^{\frac{1}{2}}g_b(\mu \rho)\sin (\phi) \\
    b_{0,r}(z,r) = 0 \\
    b_{0,z}(z,r) = 0
    \end{cases}
    \end{equation*}
    where $f_u$ and $g_b$ are as chosen in \eqref{suppfg}.
    
    For $0<s<\frac{1}{2}$, we adjust again the parameter $\gamma$ as follows: 
    \begin{equation*}\begin{cases}
    \gamma=\frac{\frac{1}{2}-s}{100}>0,\\
    \nu = \mu^{1-\gamma}.
    \end{cases}
    \end{equation*}

    Due to the dissipative terms $-\Delta \bar{u}$ and $-\Delta \bar{b}$ , the left-hand side of the energy inequality \eqref{eq:Yk_ODE} gains $+ \|\nabla^{k+1} w\|_{L^2}^2 + \|\nabla^{k+1} \beta\|_{L^2}^2$. The same calculation as in Section \ref{sec-pertur} shows that the errors generated by the added dissipation over the short time interval $[0,t^*]$ are compatible with the previous estimates in \eqref{error_L2}. Hence, the dissipation remains negligible before the norm inflation develops.

    Arguing as in Sections \ref{sec-proof} and \ref{subsec-non-resistive case}, and choosing $\mu$ sufficiently
    large depending on $\varepsilon$, we obtain
    \[
    \|w\|_{L^\infty([0,t^*];H^s)}
    +\|\beta\|_{L^\infty([0,t^*];H^s)}
    \lesssim \varepsilon^2,
    \]
    together with
    \[
    \|u\|_{L^\infty([0,t^*];H^s)
    	\cap L^2([0,t^*];H^{s+1})}
    \leq \varepsilon.
    \]
    Consequently,
    \begin{equation*}
    \begin{cases}
    \displaystyle
    \|b(t^*)\|_{H^s(\mathbb R^3)}
    \geq \dfrac{1}{\varepsilon},\\[2mm]
    \displaystyle
    \|u\|_{L^\infty([0,t^*];H^s(\mathbb R^3))
    	\cap L^2([0,t^*];H^{s+1}(\mathbb R^3))}
    \leq \varepsilon,
    \end{cases}
    \qquad 0<t^*\leq\varepsilon.
    \end{equation*}
    This completes the proof in the viscous and resistive case and
    concludes the proof of Theorem \ref{th-ill-mhd-unified}.\\

    \section*{Acknowledgments}
    W. Ye was supported by the Science and Technology Planning Project of Guangzhou (Grant No. 2025A04J0002) and the National Natural Science Foundation of China (No. 12401277, 12171493, 12271051 and 12371095). Z. Yin was partially supported by the National Natural Science Foundation of China (No.12571261).

    \section*{Declarations}
    \begin{itemize}
    	\item Competing interests: The authors declare no competing interests.
    	\item Data availability: No datasets were generated or analysed during the current study.
    	\item Author contribution: All authors contributed to writing and reviewing the manuscript.
    \end{itemize}

\end{document}